\documentclass[11pt]{article}
\usepackage{latexsym,amsmath,amssymb,amsfonts,mathrsfs,amsthm}
\usepackage{epsf,graphicx,epsfig,color,cite,cases}
\usepackage{subfigure,graphics,multirow,marginnote,enumerate,bm}
\usepackage{indentfirst}
\usepackage{bm}
\newcommand{\be}{\begin{eqnarray}}
\newcommand{\ben}{\begin{eqnarray*}}
\newcommand{\en}{\end{eqnarray}}
\newcommand{\enn}{\end{eqnarray*}}

\newtheorem{theorem}{Theorem}[section]
\newtheorem{lemma}[theorem]{Lemma}

\newtheorem{definition}[theorem]{Definition}
\begin{document}
\renewcommand{\theequation}{\arabic{section}.\arabic{equation}}

%\title[ ]
\title{\bf   Factorization method for near-field inverse scattering problems in elastodynamics}

\author{Chun Liu\thanks{School of Mathematical Sciences, Nankai University,
Tianjing 300071, China  ({\tt liuchun@nankai.edu.cn})}
\and Guanghui Hu\thanks{School of Mathematical Sciences, Nankai University,
Tianjing 300071, China ({\tt ghhu@nankai.edu.cn})}
\and Tao Yin\thanks{LSEC, Institute of Computational Mathematics and Scientific/Engineering Computing, Acad-emy of Mathematics and Systems Science, Chinese Academy of Sciences, Beijing 100190, China, andSchool of Mathematical Sciences, University of Chinese Academy of Sciences, Beijing 100049, China  ({\tt yintao@lsec.cc.ac.cn})}
\and Bo Zhang\thanks{LSEC and Academy of Mathematics and Systems Sciences, Chinese Academy of Sciences,
Beijing, 100190, China and School of Mathematical Sciences, University of Chinese Academy of Sciences,
Beijing 100049, China ({\tt b.zhang@amt.ac.cn})}}
 
\date{}

\maketitle

%\vspace{.2in}

\begin{abstract}
Consider a time-harmonic elastic point source incident on a bounded obstacle which is embedded in an open space filled with a homogeneous and isotropic elastic medium. This paper is concerned with the inverse problem of recovering the location and shape of the obstacle from near-field data generated by infinitely many incident point source waves at a fixed energy. The incident point sources and the receivers for recording scattered signals are both located on a spherical closed surface, on which an outgoing-to-incoming operator is defined for facilitating the factorization of the near-field operator. Numerical examples in 2D are presented to show the validity and accuracy of the inversion algorithm.

\vspace{.2in}

{\bf Keywords:} Factorization method, inverse scattering, near-field data, Navier equation, point sources.

\end{abstract}

\maketitle

\section{Introduction}

The obstacle scattering problem is one of the fundamental scattering problems in scattering theory. The direct obstacle scattering problem involves determining the scattered field from the known obstacle and given incident field, while the inverse obstacle scattering problem aims to reconstruct the shape and location of the obstacle from the measurement of the scattered field. It has been widely studied due to its diverse applications in many scientific areas \cite{JS1958,JH1989}, such as radar and sonar, medical imaging, and geophysical exploration \cite{JP1977,HP2008}. Unlike acoustic and electromagnetic waves, elastic wave is governed by the Navier equation, which is more complex due to the coupling of longitudinal and transverse waves that propagate at different speeds. Therefore, it presents challenges for inverse elastic obstacle scattering problem on mathematical theory and computation.

Over the past few decades, significant progress has been made in both theoretical analysis and numerical methods for the inverse elastic obstacle scattering problem. The uniqueness results using infinitely many incident plane waves have been established in \cite{peter1993uniqueness, Hu-K-S2012} with full-phase far-field data and in \cite{Cheng-Dong2024} using phaseless far-field data with a reference ball. Additionally,  uniqueness results using infinitely many incident point sources were obtained in \cite{A-K2002} with near-field data. Similar to the acoustic case, this uniqueness result with near-field data could also be easily proved by using Rellich's lemma and the mixed reciprocity relation for elastic waves. The numerical methods can be broadly classified into two types: the
quantitative method and the qualitative method. Quantitative methods, such as optimization-based iterative approaches,
include the domain derivatives \cite{Li-W-Z2016}, the continuation method \cite{Y-L-L-Y2019}. Qualitative methods, imaging based direct method, refers to the sampling method and its variants, such as the linear sampling method \cite{Arens2001}, the direct sampling method \cite{Ji-L-Y2018}, and the factorization method \cite{E-H2019}. While iterative methods require prior information about the geometry and physical properties of the obstacle and involve solving adjoint direct scattering problems at each iteration, sampling methods overcome these difficulties but may yield less accurate reconstructions compared to iterative methods.

The basic idea of the sampling methods is to design an indicator function which attains large values inside the underlying scatterer and relative small values outside. In all of the sampling methods, the measurements may be divided into two types: far-field data and near-field data. However, the theoretical analysis for far-field data is more developed than that for near-field data. For the near-field case, \cite{H-Y2014} established the factorization method for the near-field operator by introducing the outgoing-to-incoming(OtI) operator.  In their work, the OtI operator for the Helmholtz equation was constructed on a spherical surface and the factorization method worked for recovering both impenetrable and penetrable scatterers.  This approach was later extended by \cite{Y-H-X-Z2016} to the inverse fluid-solid interaction problem with non-spherical measurement surface.  Additionally, \cite{I-J-Z2012} established a direct sampling method using near-field measurements for inverse acoustic obstacle scattering, and \cite{L-M-X-Z2022} proposed a modified sampling method applicable to inverse acoustic cavity scattering problems.

In this paper, we consider an inverse elastic scattering problem of determining the shape and location of an obstacle using near-field measurement data. The obstacle is assumed to be an elastically rigid body embedded in an open space filled with a homogeneous and isotropic elastic medium.  Both the incident point sources and the receivers for measuring the scattered field are located on a spherical closed surface. Motivated by the works of \cite{H-Y2014} and \cite{Y-H-X-Z2016}, we propose a factorization method for this inverse problem.
 
The paper is organized as follows: In Section \ref{sec1}, we formulate the problem in three dimensions. Section \ref{sec2} introduces the definition of the outgoing-to-incoming operator and the factorization of the near-field operator. Finally, Section \ref{Numeric} presents numerical experiments conducted in two dimensions to validate the proposed method.

\section{Problem formulation}\label{sec1}

This paper is concerned with the time-harmonic scattering of elastic waves by a bounded rigid obstacle at a fixed frequency. Denoted by $D\subset \mathbb{R}^3$ the bounded obstacle with the $ C^2-$smooth boundary $\partial D$. The exterior $D^c:=\mathbb{R}^3 \setminus \overline{D}$, assumed to be connected, is characterized by a mass density $\rho>0$ and Lam$\acute{e}$ constants $\lambda $ and $\mu $, satisfying $\mu>0$ and $2\mu +3\lambda >0$. Without loss of generality, we set $\rho\equiv 1$ in $ \mathbb{R}^3$. Let $B_R:=\{\bm{x}=(x_1,x_2,x_3)^\top\in \mathbb{R}^3:|\bm{x}|<R\}$ be a ball of radius $R>0$ such that $\overline{D}\subset B_R$, and let $S_R:=\{\bm{x}\in \mathbb{R}^3:|\bm{x}|=R\}$ denote the boundary of  $B_R $. Here $(\cdot)^\top$ represents the transpose of a vector or a matrix. We write the unit sphere in $\mathbb{R}^3$ as $\mathbb{S}^2:= \{\bm{x}\in \mathbb{R}^3:|\bm{x}|=1\}=S_1. $

The propagation of time-harmonic elastic waves in $D^c$ is governed by the Navier equation  

\begin{equation}\label{Navier-equation}
\Delta^* \bm{u} + \omega^2 \bm{u} = \bm{0},\ \  \bm{x} \in D^c,  \qquad \ \Delta^*=\mu\Delta +(\lambda+\mu)\nabla(\nabla\cdot),
\end{equation}
where $\bm{u}$ represents the scattered field and the angular frequency $\omega>0$ is constant. The solution $\bm{u}$ of (\ref{Navier-equation}) can be decomposed into te sum of compressional and shear waves,
\begin{equation}\label{H-decomposition}
\nonumber
\bm{u}=  \bm{u}_p+ \bm{u}_s,\ \bm{u}_p=- k_p ^{-2}\nabla (\nabla\cdot \bm{u}),\ \ \bm{u}_s=- k_s ^{-2}\nabla\times(\nabla\times \bm{u}),
\end{equation}
where $ k_p:= \omega/\sqrt{2\mu +\lambda } ,\ \ k_s:= \omega/ {\sqrt{\mu}}$ are the compressional and shear wave numbers, respectively. The components $ \bm{u}_p$ and $\bm{u}_s $ satisfy the vector Helmholtz equations 
\begin{equation*}
\Delta \bm{u}_p +  k_p^2 \bm{u}_p = \bm{0},\qquad  \Delta \bm{u}_s +  k_s^2 \bm{u}_s =\bm{0},\qquad \bm{x} \in D^c. 
\end{equation*}
Since the obstacle is elastically rigid, we have the inhomogeneous Dirichlet boundary condition 
\begin{equation}\label{boundary}
\bm{u}=-\bm{u}^{in},  \qquad \bm{x}\in \partial D,
\end{equation}
where $ \bm{u}^{in} $ is the incident field. The surface stress (or traction) operator $T$ on the boundary $\partial D$ is defined as 
\begin{equation}\label{surface operator}
\nonumber
  T\bm{u}= 2 \mu  \bm{n} \cdot \nabla\bm{u}+\lambda  \bm{n}\nabla\cdot \bm{u} + \mu   \bm{n} \times \nabla\times \bm{u},
\end{equation}
where $ \bm{n}$ is the unit normal vector on $\partial D$ %at the point $\bm{x}
directed into the exterior of $D$.

 In this work, we consider the point source incident wave $\bm{u}^{in}(\bm{x},\bm{y};\mathbf{a})$  %generated at the source position $\bm{y} \in D^c$
  with the polarization direction $\bm{a}\in  \mathbb{S}^2$ 
 \begin{equation*}
 \bm{u}^{in}(\bm{x},\bm{y}; \bm{a})=\Pi(\bm{x},\bm{y})\bm{a},\qquad \bm{x}\neq \bm{y},\ \bm{y}\in D^c,
 \end{equation*}
where $ \Pi(\bm{x},\bm{y}) $ is the free-space Green's tensor to the Navier equation
\begin{equation*}
 \Pi(\bm{x},\bm{y})=\frac{1}{\mu}\Phi_{k_s}(\bm{x},\bm{y})\mathbf{I}+\frac{1}{\omega^2}\nabla_{\bm{x}}\nabla_{\bm{x}}^\top[\Phi_{k_s}(\bm{x},\bm{y})-\Phi_{k_p}(\bm{x},\bm{y})],\qquad \bm{x}\neq \bm{y}.
 \end{equation*}
Here, $\mathbf{I}$ stands for the $3\times 3$ identity matrix, $ \Phi_{k}(\bm{x},\bm{y})=\exp(ik|\bm{x}-\bm{y}|)/(4\pi|\bm{x}-\bm{y}|) $ is the free-space fundamental solution to the Helmholtz equation $(\Delta + k^2) u=0$ in $ \mathbb{R}^3$. The scattered field corresponding to $\bm{u}^{in}(\bm{x},\bm{y}; \bm{a}) $ is denoted as $ \bm{u}=\bm{u}(\bm{x},\bm{y}; \bm{a})$. 
  
The scattered field $\bm{u}$ is required to satisfy the Kupradze radiation condition
\begin{align}\label{radiation}
 \lim_{r=|\bm{x}|\rightarrow \infty }r\left[\frac{\partial \bm{u}_\alpha(\bm{x})}{\partial r}-i k_\alpha\mathbf{ u}_\alpha(\bm{x})\right]=0,\qquad \alpha=p,s,
\end{align}
uniformly in all directions $\hat{\bm{x}}=\bm{x}/{|\bm{x}|}\in \mathbb{S}^2$. This condition (\ref{radiation}) is known as the Sommerfeld condition for the compressional and shear parts of $\bm{u} $. The radiation conditions in (\ref{radiation}) lead to the P-part $\bm{u}_p^{\infty}$ and S-part $\bm{u}_s^{\infty}$ of the far-field pattern of $\bm{u}$, given by the asymptotic behavior
\begin{equation*}\label{far}
\bm{u}(\bm{x})=\frac{\mathrm{exp}(ik_p |\bm{x}|)}{|\bm{x}|}\bm{u}_p^{\infty}(\hat{\bm{x}})+\frac{\mathrm{exp}(ik_s |\bm{x}|)}{|\bm{x}|}\bm{u}_s^{\infty}(\hat{\bm{x}})+\mathcal{O}\left(\frac{1}{|\bm{x}|^2}\right),\ |\bm{x}|\rightarrow \infty,
\end{equation*}
where $\bm{u}_p^{\infty}$ and $\bm{u}_s^{\infty}$ are the far-field part of $\bm{u}_p $ and $\bm{u}_s $, respectively. In this paper, we define the far-field pattern $\bm{u}^{\infty}$ of the scattered field $\bm{u}$ as the sum of $\bm{u}_p^{\infty}$ and $\bm{u}_s^{\infty}$, which means $\bm{u}^{\infty}:= \bm{u}_p^{\infty}+\bm{u}_s^{\infty} $.

Given the incident field $ \bm{u}^{in}$, the \textit{direct scattering problem} (DP) for (\ref{Navier-equation})-(\ref{radiation}) is to find the scattered field $\bm{u}$ for a known obstacle $D$. By using variational approach in combination with properties of transparent operator for the Navier equation (\ref{Navier-equation}), \cite{Li-Y2019} shows that the problem (DP) admits at most one solution $\bm{u}\in [H_{loc}^1( D^c )]^3$. The inverse scattering problem is to reconstruct the shape and location of the obstacle $D$ from the near-field data
$\{\bm{u}(\bm{x},\bm{y};\bm{a_i}): \bm{x},\bm{y}\in S_R, \bm{a_i}\in \mathbb{S}^2,i=1,2,3\}$, where $\bm{a_i}(i=1,2,3)$ are three linearly independent polarizations. Theorem 3 and Theorem 4 in \cite{peter1993uniqueness} prove that the far-field data for incident plane waves $ \bm{u}^{in}(\bm{x})=\bm{d}\exp(ik_p\bm{d}\cdot\bm{x})+\bm{q}\exp(ik_s\bm{d}\cdot\bm{x})$ for all $\bm{d},\bm{q}\in \mathbb{S}^2, \bm{d}\cdot\bm{q}=0 $ at a fixed frequency $\omega$ can uniquely determine the scatterer $D.$ By using Rellich's lemma and the mixed reciprocity relation (see \cite{XY}), we can obtain that the near-field data for incident point source can uniquely reconstruct the obstacle $D.$

\section{Factorization of near-field operator} \label{sec2}

For $ \bm{g}\in [L^2(S_R)]^3 $, we define the incident field $\bm{v}_{\bm{g}}^{ in }$ as 
\begin{equation}\label{new-incident}
 \bm{v}_{\bm{g}}^{ in }(\bm{x})=\int_{S_R}\Pi(\bm{x},\bm{y})\bm{g}(\bm{y}){\rm{d}}s(\bm{y}),\qquad \bm{x}\in  \mathbb{R}^3.
\end{equation}
The scattered field corresponding to $\bm{v}_{\bm{g}}^{ in }$ takes the form
\begin{equation}\label{new-scattered}
 \bm{v}_{\bm{g}}(\bm{x})=\int_{S_R}\bm{u} (\bm{x},\bm{y};\bm{g}(\bm{y})){\rm{d}}s(\bm{y}),\qquad \bm{x}\in D^c,
\end{equation}
where $\bm{u} (\bm{x},\bm{y};\bm{g}(\bm{y}))$ represents the scattered field generated by incident field $\Pi(\bm{x},\bm{y})\bm{g}(\bm{y}).$ Then we can define the near-field operator  $N: [L^2(S_R)]^3\rightarrow [L^2(S_R)]^3$ by
\begin{equation}\label{Near-field-operator}
(N\bm{g})(\bm{x})=\int_{S_R}\bm{u} (\bm{x},\bm{y};\bm{g}(\bm{y})){\rm{d}}s(\bm{y}),    \qquad \bm{x}\in S_R.
\end{equation}
Clearly, $N\bm{g}$ is the restriction to $S_R$ of the scattered field $\bm{v}_{\bm{g}}$ generated by the incident wave $\bm{v}_{\bm{g}}^{ in }$.

In this section, we will establish a factorization of the near-field operator $N$ for incident point sources. Motivated by \cite{H-Y2014} in acoustics, we introdunce the outgoing-to-incoming operator $\mathcal{T}$ and factorization of near-field operator inorder to get a symmetric factorization of $N$.

\subsection{Solution operator}\label{sec-solution}
\begin{definition}
Given $\bm{h}\in [H^{\frac{1}{2}}(\partial D)]^3$, denoted by $\bm{v}\in [H_{loc}^1( D^c )]^3$ the unique solution to the boundary value problem
\begin{align*}
\Delta^* \bm{v}  + \omega^2 \bm{v}  = \bm{0},\ \  \bm{x} \in D^c,\\
\bm{v} =\bm{h},  \qquad \bm{x}\in \partial D,
\end{align*}
and that $\bm{v} $ satisfies the Kupradze radiation condition (\ref{radiation}). The solution operator is defined as  $G: [H^{\frac{1}{2}}(\partial D)]^3\rightarrow [L^2(S_R)]^3$ 
\begin{equation}\label{solution-operator}
G \bm{h}=\bm{v}|_{S_R}.
\end{equation}
\end{definition}

Clearly, $G$ maps the boundary value of a radiating solution on $\partial D$ onto the near-field pattern on $S_R$.
\begin{theorem}\label{G-com}
The solution operator $G: [H^{\frac{1}{2}}(\partial D)]^3\rightarrow [L^2(S_R)]^3$ is compact, injective with a dense range in $  [L^2(S_R)]^3 $.
\end{theorem}
\begin{proof}
The injectivity of $G$ follows form the uniqueness of the exterior Dirichlet problem \cite{peter1993uniqueness} and the analytic continuation argument. Since the scattering problem (\ref{Navier-equation})-(\ref{radiation}) has a solution $\bm{u}\in [H_{loc}^1( D^c )]^3$, and $[H^{\frac{1}{2}}( S_R)]^3 $ can be compactly embedded into $[L^2(S_R)]^3$, $G$ is compact.

Denote $\bm{u}$ the unique solution to the scattering problem (\ref{Navier-equation})-(\ref{radiation}) with the boundary value
$ \bm{u}=\bm{f}\in [H^{\frac{1}{2}}(\partial D)]^3$. Then $\bm{u}$ can be expressed as 
\begin{equation}\label{formula}
\bm{u}(\bm{x})=\int_{\partial D}\{ T\Pi(\bm{x},\bm{y}) \bm{u}(\bm{y})-\Pi(\bm{x},\bm{y}) T\bm{u}(\bm{y})\}{\rm{d}}s(\bm{y}),\qquad \bm{x}\in  D^c.
\end{equation}
Then one has for $ \bm{g}\in [L^2(S_R)]^3 $ that
\begin{equation}\label{green1}
\begin{aligned}
(G\bm{f},\bm{g})_{[L^2(S_R)]^3}&=\int_{ S_R }\ \bm{u}(\bm{x})\cdot\overline{\bm{g}}(\bm{x}){\rm{d}}s(\bm{x})\\
&=\int_{\partial D}\{ \bm{u}(\bm{y})\cdot T \bm{v}_{\overline{\bm{g}}}^{in}(\bm{y}) - T\bm{u}(\bm{y}) \cdot  \bm{v}_{\overline{\bm{g}}}^{in}(\bm{y}) \}{\rm{d}}s(\bm{y}).
\end{aligned}
\end{equation}
As both $\bm{u}$ and $\bm{v}_{\overline{\bm{g}}}$ are radiating solutions, we find
\begin{equation}\label{green2}
\begin{aligned}
 \int_{\partial D}\{ \bm{u}(\bm{y})\cdot T \bm{v}_{\overline{\bm{g}}} (\bm{y}) - T\bm{u} (\bm{y})\cdot\bm{v}_{\overline{\bm{g}}} (\bm{y}) \}{\rm{d}}s(\bm{y})=0.
\end{aligned}
\end{equation}
Using the boundary condition $\bm{v}_{\overline{\bm{g}}}^{in}+\bm{v}_{\overline{\bm{g}}}=\bm{0}$, (\ref{green1}) and (\ref{green2}), we obtain
\begin{align*}
(G\bm{f},\bm{g})_{[L^2(S_R)]^3}&=\int_{\partial D}\  \bm{u}(\bm{y})\cdot T\left( \bm{v}_{\overline{\bm{g}}}^{in}(\bm{y})+\bm{v}_{\overline{\bm{g}}}(\bm{y})\right) {\rm{d}}s(\bm{y})\\
&=\int_{\partial D}\  \bm{f}(\bm{y})\cdot T\left( \bm{v}_{\overline{\bm{g}}}^{in}(\bm{y})+\bm{v}_{\overline{\bm{g}}}(\bm{y})\right) {\rm{d}}s(\bm{y}).
\end{align*}
%\begin{equation*}
%(G\bm{f},\bm{g})_{[L^2(S_R)]^3}=\int_{\partial D}\  \bm{u}%(\bm{y})\cdot T\left( \bm{v}_{\overline{\bm{g}}}^{in}(\bm{y})+%\bm{v}_{\overline{\bm{g}}}(\bm{y})\right) {\rm{d}}s(\bm{y}).
%\end{equation*}
As a result, the adjoint operator $G^*$ can be characterized as $G^* \bm{g} =\overline{T\left( \bm{v}_{\overline{\bm{g}}}^{in} +\bm{v}_{\overline{\bm{g}}} \right)}$. If $ G^* \bm{g}=\bm{0}$, then $  T\left( \bm{v}_{\overline{\bm{g}}}^{in} +\bm{v}_{\overline{\bm{g}}} \right)=\bm{0}$, which implies that $ T  \bm{v}_{\overline{\bm{g}}}= - T \bm{v}_{\overline{\bm{g}}}^{in}$ on $\partial D$. By virtue of (\ref{formula}) and boundary condition  $\bm{v}_{\overline{\bm{g}}}^{in}=-\bm{v}_{\overline{\bm{g}}}=\bm{0}$ on $\partial D$, one finds that
\begin{equation} 
\nonumber
\bm{v}_{\overline{\bm{g}}}(\bm{x})=-\int_{\partial D}\{ T\Pi(\bm{x},\bm{y}) \cdot\bm{v}_{\overline{\bm{g}}}^{in}(\bm{y})-\Pi(\bm{x},\bm{y}) \cdot T\bm{v}_{\overline{\bm{g}}}^{in}(\bm{y})\}{\rm{d}}s(\bm{y}), \qquad \bm{x}\in  D^c.
\end{equation}
Since both $\bm{v}_{\overline{\bm{g}}}^{in}$ and $\bm{y}\rightarrow\Pi(\bm{x},\bm{y})$ for $\bm{x}\in  D^c $ satisfy the Navier equation in $D$, the Betti’s formula applied over $D$ shows that
\begin{align*}
-\bm{v}_{\overline{\bm{g}}}(\bm{x})=&\int_{\partial D}\{ T\Pi(\bm{x},\bm{y})  \bm{v}_{\overline{\bm{g}}}^{in}(\bm{y})-\Pi(\bm{x},\bm{y})   T\bm{v}_{\overline{\bm{g}}}^{in}(\bm{y})\}{\rm{d}}s(\bm{y})\\
=&\int_{D}\{\Delta^*\Pi(\bm{x},\bm{y}) \bm{v}_{\overline{\bm{g}}}^{in}(\bm{y})-\Pi(\bm{x},\bm{y}) \Delta^*\bm{v}_{\overline{\bm{g}}}^{in}(\bm{y})\}{\rm{d}} \bm{y}\\
=&\int_{D}\{(\Delta^*\Pi(\bm{x},\bm{y})+\omega^2\Pi(\bm{x},\bm{y})) \bm{v}_{\overline{\bm{g}}}^{in}(\bm{y})-\Pi(\bm{x},\bm{y}) (\Delta^*\bm{v}_{\overline{\bm{g}}}^{in}(\bm{y})+\omega^2\bm{v}_{\overline{\bm{g}}}^{in}(\bm{y}))\}{\rm{d}} \bm{y}\\
=&\bm{0},
\end{align*}
for all $\bm{x}\in D^c $. Thus $\bm{v}_{\overline{\bm{g}}}(\bm{x})= T\bm{v}_{\overline{\bm{g}}}(\bm{x})=\bm{0}$ on $\partial D$, which means that $\bm{v}^{in}_{\overline{\bm{g}}}(\bm{x})= T\bm{v}^{in}_{\overline{\bm{g}}}(\bm{x})=\bm{0}$ on $\partial D$. Using the Holmgren's uniqueness theorem, one finds that $\bm{v}_{\overline{\bm{g}}}^{in}=\bm{0}$ in
$  \mathbb{R}^3$. The jump relation shows that $\bm{g}=\mathcal{T}(\bm{v}_{\overline{\bm{g}}}^{in})^{-}-\mathcal{T}(\bm{v}_{\overline{\bm{g}}}^{in})^{+}$, where the superscripts $'-'$ and $'+'$ denote respectively the limits from inside and outside of $B_R$. Thus $\mathbf{g}=\bm{0},$ which means that $G^*$ is injective. Hence $G$ has dense range by Theorem 4.6 in\cite{Colton2013}.
\end{proof}

Define the incidence operator $H: [L^2(S_R)]^3\rightarrow [H^{\frac{1}{2}}(\partial D)]^3$ by
\begin{equation*}\label{incidence-operator}
(H\bm{g})(\bm{x})=\int_{S_R}\Pi(\bm{x},\bm{y}) \bm{g}(\bm{y}){\rm{d}}s(\bm{y}), \qquad \bm{x}\in \partial D.
\end{equation*}
Since $\Pi(\bm{x},\bm{y})=\Pi(\bm{x},\bm{y})^\top$, the $L^2$ adjoint $H^*:[H^{-\frac{1}{2}}(\partial D)]^3\rightarrow [L^2(S_R)]^3$ of $H$ is given by
\begin{equation*}\label{incidence-op-adjoint}
(H^*\bm{\varphi})(\bm{x})=\int_{\partial D}\overline{\Pi(\bm{x},\bm{y})} \bm{\varphi}(\bm{y}){\rm{d}}s(\bm{y}), \qquad \bm{x}\in S_R.
\end{equation*}
Here, for a matrix $A=(a_{ij})$, we define $\overline{A}:=(\overline{a_{ij}})$.
 From the definition of $N, G$ and $H$, the following relation holds:
\begin{equation}\label{factorization1}
N=-GH.
\end{equation}

\subsection{Outgoing-to-incoming operator}
In this subsection we give the definition of OtI operator on a spherical surface. Denote $(r,\theta,\varphi) $ the spherical coordinates of $ \bm{x}\in \mathbb{R}^3$, where $\theta\in [0,2\pi]$
and $\varphi\in [0, \pi]$ are the Euler angles. Let
\begin{align*}
\widehat{\bm{r}}&= (\sin \theta\cos\varphi  ,\sin \theta\sin \varphi,\cos \theta)^\top,\\
\widehat{\bm{\theta}}&= (\cos \theta \cos \varphi,\cos\theta\sin  \varphi,-\sin \theta)^\top,\\
\widehat{\bm{\varphi}}&= (-\sin \varphi , \cos \varphi, 0)^\top 
\end{align*}
be the unit vectors in the spherical coordinates. $ \{\widehat{\bm{r}},\widehat{\bm{\theta}},\widehat{\bm{\varphi}}\}$ forms a local orthonormal basis, and $\widehat{\bm{r}}$ is also the unit outward normal vector on $S_R$.

Let $\{Y_n^m(\theta,\varphi):n=0,1,2\ldots,m=-n.\ldots n\}$ be the orthonormal sequence of spherical harmonic functions of order $n$ on the unit sphere. Define the vector spherical harmonics:
\begin{align*}
\bm{U}_n^m(\widehat{\bm{x}})&:=\frac{1}{\sqrt{\delta_n}}\nabla_{ \mathbb{S}^2}Y_n^m(\widehat{\bm{x}})  ,\\
\bm{V}_n^m(\widehat{\bm{x}})&:=  \widehat{\bm{x}} \times \bm{U}_n^m(\widehat{\bm{x}}),\\
\bm{W}_n^m(\widehat{\bm{x}})&:= Y_n^m(\widehat{\bm{x}})\  \widehat{\bm{x}}, 
\end{align*}
where $ \widehat{\bm{x}} :=\bm{x}/|\bm{x}|\in\mathbb{S}^2$ and $\delta_n:=n(n+1), n=0,1,2\ldots,m=-n\ldots n.$ Denote by $ \nabla_{ \mathbb{S}^2} $ the surface gradient on $ \mathbb{S}^2. $  Set $M_n^m(\widehat{\bm{x}}):=(\bm{V}_n^m(\widehat{\bm{x}}),\bm{U}_n^m(\widehat{\bm{x}}),\bm{W}_n^m(\widehat{\bm{x}}))$. We can easily show that $\{M_n^m(\theta,\varphi):n=0,1,2\ldots,m=-n.\ldots n\}$ form a complete orthonormal system in $ [L^2(\mathbb{S}^2)]^3.$

The radiating solution $\bm{u}$ to the Navier equation (\ref{Navier-equation}) in 
$|\bm{x}|\geq R$ can be expressed as (see \cite[Section 2.4]{B-H-S-Y2018})
 
\begin{equation}\label{u-expansion}
\begin{aligned}
\bm{u}(\bm{x})=\sum_{n=0}^{\infty}\sum_{m=-n}^{n}M_n^m(\widehat{\bm{x}})A_n(|\bm{x}|){\bm{\alpha}_n^m}  
\end{aligned}
\end{equation}
where $\bm{\alpha}_n^m \in \mathbb{C}^3$ and $A_n(|\bm{x}|)$ is defined as 
\begin{equation}\label{coefficient-matrix}
\begin{aligned}
A_n(|\bm{x}|):=\begin{pmatrix}
h_n^{(1)}(k_s|\bm{x}|)&0&0\\
0&-\left( h_n^{(1)}(k_s|\bm{x}|)+k_s|\bm{x}|h_n^{(1)'}(k_s|\bm{x}|)\right)/|\bm{x}| & \sqrt{\delta_n}h_n^{(1)}(k_p|\bm{x}|)/|\bm{x}|\\
0&-\sqrt{\delta_n}h_n^{(1)}(k_s|\bm{x}|)/|\bm{x}|&k_p h_n^{(1)'}(k_p|\bm{x}|)
\end{pmatrix}
\end{aligned}.
\end{equation}

Here, $h_n^{(1)}$ and $h_n^{(1)'}$ are the spherical Hankel functions of order $n$ and its first derivative, respectively.

\begin{definition}
Let $\bm{u}$ be an outgoing solution to the problem  (\ref{Navier-equation})-(\ref{radiation}). Suppose $\bm{u}$ admits the expansion (\ref{u-expansion}) for all $|\bm{x}|\geq R$. Then the OtI mapping $\mathcal{T}: \text{Range}(G) \rightarrow [L^2(S_R)]^3 $ is defined as
\begin{equation}
\mathcal{T}(\bm{u}|_{S_R})=\tilde{\bm{u}}|_{S_R}
\end{equation}
with
\begin{equation}\label{u-expansion1}
\begin{aligned}
\tilde{\bm{u}}(\bm{x}):=-\sum_{n=0}^{\infty}\sum_{m=-n}^{n}M_n^m(\widehat{\bm{x}}) \overline{A_n(|\bm{x}|)}{\bm{\alpha}_n^m} ,\qquad |\bm{x}|\geq R.
\end{aligned}
\end{equation}
Here $A_n(|\bm{x}|)$ is defined in (\ref{coefficient-matrix}).
\end{definition}
By definition, the operator $\mathcal{T}$ is linear, bounded and one-to-one. As $\bm{u}$ is an outgoing wave satisfying the Kupradze radiation condition (\ref{radiation}), $\tilde{\bm{u}}$ is an incoming wave satisfying the Navier equation (\ref{Navier-equation}) and incoming condition
\begin{align*}
 \lim_{r=|\bm{x}|\rightarrow \infty }r\left[\frac{\partial \tilde{\bm{u}}_\alpha(\bm{x})}{\partial r}+i k_\alpha\tilde{\bm{u}}_\alpha(\bm{x})\right]=0,\qquad \alpha=p,s.
\end{align*} 

\begin{theorem}
The OtI mapping $\mathcal{T}: \text{Range}(G) \rightarrow [L^2(S_R)]^3 $ takes the form
\begin{equation}\label{OtI}
(\mathcal{T}\bm{\varphi})(\bm{x})=-\frac{1}{R^2}\sum_{n=0}^{\infty}\sum_{m=-n}^{n}M_n^m(\widehat{\bm{x}})\overline{  A_n(R)} \left[  A_n(R)\right]^{-1}\int_{S_R}  \overline{M_n^m(\widehat{\bm{y}})}^\top \bm{\varphi}(\bm{y}){\rm{d}}s(\bm{y}).
\end{equation}
\end{theorem}
\begin{proof}
For each $\bm{\varphi} \in \text{Range}(G)$, we have the expansion
\begin{equation}\label{fi-expansion}
\begin{aligned}
\bm{\varphi}(\bm{x})=\sum_{n=0}^{\infty}\sum_{m=-n}^{n}  M_n^m(\widehat{\bm{x}})A_n(R){\bm{\varphi}_n^m}
\end{aligned},\qquad |\bm{x}|= R,
\end{equation}
where $\bm{\varphi}_n^m\in \mathbb{C}^3$ and $A_n(R)$ is defined in (\ref{coefficient-matrix}). Clearly, 
\begin{align*}
 A_n(R){\bm{\varphi}_n^m} =\frac{1}{R^2}\int_{S_R}    \overline{M_n^m(\widehat{\bm{y}})}^\top\bm{\varphi}(\bm{y}){\rm{d}}s(\bm{y}). 
\end{align*}
Introduce the matrix $C_n(R):=  A_n(R) Q_n(R)$, where the matrix $Q_n(R)$ is defined as 
\begin{equation*}
\begin{aligned}
Q_n(R):=\begin{pmatrix}
R/h_n^1( k_sR)&0&0\\
0&R/h_n^1( k_sR) &0 \\
0&0&R/h_n^1( k_pR) 
\end{pmatrix}.
\end{aligned}
\end{equation*}
Lemma 2.15 in \cite{B-H-S-Y2018} shows that the matrix $C_n(R)$ is invertible for all $n\geq 0, R>0$. Since $h_n^{(1)}(kR)\neq0$ for all $n>0, k=k_p,k_s,$   the matrix $Q_n(R)$ is invertible. Then $   A_n(R)  $ is invertible which means that
\begin{align}\label{coefficient}
{\bm{\varphi}_n^m}=\frac{1}{R^2}\left[  A_n(R)  \right]^{-1}\int_{S_R} \overline{M_n^m(\widehat{\bm{y}})}^\top \bm{\varphi}(\bm{y}){\rm{d}}s(\bm{y}).
\end{align}
By the definition of $\mathcal{T}$, it follows that
\begin{equation} 
\nonumber
(\mathcal{T}\bm{\varphi})(\bm{x})=-\frac{1}{R^2}\sum_{n=0}^{\infty}\sum_{m=-n}^{n}M_n^m(\widehat{\bm{x}})\overline{  A_n(R)} \left[  A_n(R)\right]^{-1}\int_{S_R}  \overline{M_n^m(\widehat{\bm{y}})}^\top \bm{\varphi}(\bm{y}){\rm{d}}s(\bm{y}).
\end{equation} 
The proof is complete.
\end{proof}

Next, we consider the expansion of the fundamental solution $\Pi(\bm{x},\bm{y})$. Recall (see \cite[Theorem 2.11, 6.29]{Colton2013}) that $ \Phi_{k}(\bm{x},\bm{y}) $ and $\Phi_{k}(\bm{x},\bm{y})\mathbf{I}$ have the expansions
\begin{equation}\label{f1-expansion}
 \Phi_{k}(\bm{x},\bm{y})=ik\sum_{n=0}^{\infty}\sum_{m=-n}^{n} h_n^{(1)}(k|\bm{y}|)Y_n^m(\widehat{\bm{y}})j_n(k|\bm{x}|)\overline{Y_n^m(\widehat{\bm{x}})} 
 \end{equation} 
and 
\begin{equation}\label{f2-expansion}
\begin{aligned}
\Phi_{k}(\bm{x},\bm{y})\mathbf{I}&=ik\sum_{n=1}^{\infty}\sum_{m=-n}^{n}j_n(k|\bm{x}|)\overline{\bm{V}_n^m(\widehat{\bm{x}})}h_n^{(1)}(k|\bm{y}|)\bm{V}_n^m(\widehat{\bm{y}})^{\top}\\
&+ ik^ {-1}\sum_{n=1}^{\infty}\sum_{m=-n}^{n}\nabla\times\left(j_n(k|\bm{x}|)\overline{\bm{V}_n^m(\widehat{\bm{x}})}\right)\nabla\times\left(h_n^{(1)}(k|\bm{y}|)\bm{V}_n^m(\widehat{\bm{y}})\right)^{\top}\\
&+ik^ {-1}\sum_{n=0}^{\infty}\sum_{m=-n}^{n}\nabla \left(j_n(k|\bm{x}|)\overline{Y_n^m(\widehat{\bm{x}})}\right)\nabla \left(h_n^{(1)}(k|\bm{y}|)Y_n^m(\widehat{\bm{y}})\right)^{\top},
\end{aligned}
\end{equation}
which converge absolutely and uniformly on compact subsets of $|\bm{x}|<|\bm{y}|.$ Here $j_n$ denotes the spherical Bessel function of order $n$.

From the relation $\nabla_{\bm{x}}\nabla_{\bm{x}}\Phi_{k}(\bm{x},\bm{y})=-\nabla_{\bm{\mathbf{y}}}\nabla_{\bm{x}}\Phi_{k}(\bm{x},\bm{y})$ and $Y_n^m=\overline{Y_n^{-m}}$, we have
\begin{equation}\label{f3-expansion}
\begin{aligned}
\nabla_{\bm{x}}\nabla_{\bm{x}}^\top\Phi_{k}(\bm{x},\bm{y})&=-ik\sum_{n=0}^{\infty}\sum_{m=-n}^{n} \nabla \left(j_n(k|\bm{x}|)\overline{Y_n^m(\widehat{\bm{x}})}\right)\nabla \left(h_n^{(1)}(k|\bm{y}|){Y_n^m(\widehat{\bm{y}})}\right)^{\top}\\
&=-ik\sum_{n=0}^{\infty}\sum_{m=-n}^{n} \nabla \left(j_n(k|\bm{x}|) Y_n^m(\widehat{\bm{x}}) \right)\nabla \left(h_n^{(1)}(k|\bm{y}|){\overline{Y_n^m(\widehat{\bm{y}})}}\right)^{\top}
\end{aligned}
 \end{equation} 
Combining equations (\ref{f1-expansion})-(\ref{f3-expansion}), we obtain the expansion of $\Pi$ as follows
\begin{equation}\label{pi-expansion}
\begin{aligned}
 \Pi(\bm{x},\bm{y})&=i\sum_{n=0}^{\infty}\sum_{m=-n}^{n} \overline{ M_n^m(\widehat{\bm{x}})} B_n(|\bm{x}|)\ E\ [ A_n(|\bm{y}|)]^{\top}  M_n^m(\widehat{\bm{y}}) ^\top\\
&=i\sum_{n=0}^{\infty}\sum_{m=-n}^{n}  M_n^m(\widehat{\bm{x}}) B_n(|\bm{x}|)\ E\  [ A_n(|\bm{y}|)]^{\top} \overline{ M_n^m(\widehat{\bm{y}})} ^\top ,
\end{aligned}
 \end{equation} 
which converge absolutely and uniformly on compact subsets of $|\bm{x}|<|\bm{y}|.$ Here, $ B_n(|\bm{x}|)$ and $E$ are defined as 
\begin{equation}\label{jn-matrix}
\begin{aligned}
B_n(|\bm{x}|):=\begin{pmatrix}
j_n (k_s|\bm{x}|)&0&0\\
0&-\left(j_n(k_s|\bm{x}|)+k_s|\bm{x}|j_n'(k_s|\bm{x}|)\right)/|\bm{x}| &k_p j_n'(k_p|\bm{x}|)\\
0&-\sqrt{\delta_n}j_n(k_s|\bm{x}|)/|\bm{x}|&\sqrt{\delta_n}j_n(k_p|\bm{x}|)/|\bm{x}| 
\end{pmatrix}
\end{aligned}
\end{equation}
and
\begin{equation} 
\begin{aligned}
E:=\begin{pmatrix}
 k_s/\mu&0&0\\
0&1/(k_s\mu) &0 \\
0&0&1/(k_p(\lambda+2\mu)) 
\end{pmatrix}.
\end{aligned}
\end{equation}

\begin{lemma}\label{T-character}
$\mathcal{T}\left(\Pi(\cdot,\bm{z}) \bm{a}|_{S_R}\right)=\overline{\Pi(\cdot,\bm{z})} \bm{a}|_{S_R}$ for $|\bm{z}|<R$ and any constant $\bm{a}\in \mathbb{S}^2$.
\end{lemma}
\begin{proof}
By virtue of $\Pi(\bm{x},\bm{z})=[\Pi(\bm{x},\bm{z})]^\top$ and (\ref{pi-expansion}), we have for $|\bm{x}|=R$ that
\begin{equation} \label{expansion-z}
\begin{aligned}
 \Pi(\bm{x},\bm{z})&=i\sum_{n=0}^{\infty}\sum_{m=-n}^{n}  M_n^m(\widehat{\bm{x}})   A_n(R)   E^\top\left[ B_n(|\bm{z}|) \right]^\top    \overline{M_n^m(\widehat{\bm{z}})} ^\top\\
 &=i\sum_{n=0}^{\infty}\sum_{m=-n}^{n} \overline{ M_n^m(\widehat{\bm{x}})}   A_n(R)   E^\top\left[ B_n(|\bm{z}|) \right]^\top    M_n^m(\widehat{\bm{z}}) ^\top.
 \end{aligned}
\end{equation} 
Thus by the definition (\ref{OtI}) of $\mathcal{T}$, we have
\begin{equation}\label{expansion-z1}
\begin{aligned}
\mathcal{T}\left(\Pi(\cdot,\bm{z})  \bm{a}|_{S_R}\right)=
-\frac{1}{R^2}\sum_{n=0}^{\infty}\sum_{m=-n}^{n}  M_n^m(\widehat{\bm{x}}) \overline{  A_n(R) } \left[  A_n(R)  \right]^{-1} \int_{S_R}   \overline{M_n^m(\widehat{\bm{y}})}^{\top}\Pi(\bm{y},\bm{z})  \bm{a}{\rm{d}}s(\bm{y}).
\end{aligned}
\end{equation}
Using the expansion of $  \Pi(\bm{x},\bm{z}) $ in (\ref{expansion-z}), we have
\begin{equation} \label{expansion-z2}
\begin{aligned}
 &\int_{S_R}   \overline{M_n^m(\widehat{\bm{y}})}^{\top}\Pi(\bm{y},\bm{z})  \bm{a}{\rm{d}}s(\bm{y})\\
 =& i\sum_{l=0}^{\infty}\sum_{q=-l}^{l} \int_{S_R}   \overline{M_n^m(\widehat{\bm{y}})}^{\top}  M_l^q(\widehat{\bm{y}})    A_l(R)  E^\top\left[ B_l(|\bm{z}|) \right]^\top    \overline{M_l^q(\widehat{\bm{z}})} ^\top \bm{a}{\rm{d}}s(\bm{y})\\
 =& i\sum_{l=0}^{\infty}\sum_{q=-l}^{l} \int_{S_R}   \overline{M_n^m(\widehat{\bm{y}})}^{\top}  M_l^q(\widehat{\bm{y}}) {\rm{d}}s(\bm{y}) A_l(R)  E^\top\left[ B_l(|\bm{z}|) \right]^\top    \overline{M_l^q(\widehat{\bm{z}})}  ^\top\bm{a},
\end{aligned}
\end{equation}
where $l,q\in  \mathbb{Z}$. Since $\{M_n^m(\theta,\varphi):n=0,1,2\ldots,m=-n.\ldots n\}$ form a complete orthonormal system in $ [L^2(\mathbb{S}^2)]^3,$ we know
that
\begin{equation}\label{expansion-z3}
 \int_{S_R}\overline{M_n^m(\widehat{\bm{y}})}^{\top}  M_l^q(\widehat{\bm{y}}){\rm{d}}s(\bm{y}) =\begin{cases}
 R^2\mathbf{I},\qquad l=n,\ q=m,\\
 \bm{0},\qquad \qquad \text{else}.
 \end{cases}
\end{equation}
By substituting (\ref{expansion-z3}) into (\ref{expansion-z2}), we obtain
\begin{equation} \label{expansion-z4}
\begin{aligned}
  \int_{S_R}   \overline{M_n^m(\widehat{\bm{y}})}^{\top} \Pi(\bm{y},\bm{z})  \bm{a}{\rm{d}}s(\bm{y})=iR^2  A_n(R)   E^\top\left[ B_n(|\bm{z}|) \right]^\top   \overline{M_n^m(\widehat{\bm{z}})}^{\top} \bm{a}.
\end{aligned}
\end{equation}
Inserting (\ref{expansion-z4}) into (\ref{expansion-z1}) gives
\begin{equation}\label{expansion-z5}
\begin{aligned}
\mathcal{T}\left(\Pi(\cdot,\bm{z})  \bm{a}|_{S_R}\right)=
-i\sum_{n=0}^{\infty}\sum_{m=-n}^{n}   M_n^m(\widehat{\bm{x}})   \overline{  A_n(R) }   E^\top\left[ B_n(|\bm{z}|) \right]^\top   \overline{M_n^m(\widehat{\bm{z}})}^\top\bm{a}.
\end{aligned}
\end{equation}

As matrixes $E$ and $B_n$ are real-valued, it's easy to find that
$\mathcal{T}\left(\Pi(\cdot,\bm{z})  \bm{a} \right)=\overline{\Pi(\cdot,\bm{z})} \bm{a}$.
\end{proof}

\begin{lemma}
The adjoint operator $\mathcal{T}^*: [L^2(S_R)]^3 \rightarrow [L^2(S_R)]^3  $ is
\begin{equation}\label{OtI*}
(\mathcal{T}^*\bm{\varphi})(\bm{x}):=-\frac{1}{R^2}\sum_{n=0}^{\infty}\sum_{m=-n}^{n}   M_n^m(\widehat{\bm{x}})\left[\overline{  A_n(R) }  ^{-1}\right]^\top \left[  A_n(R) \right]^\top  \int_{S_R}   \overline{ M_n^m(\widehat{\bm{y}})}^\top\bm{\varphi}(\bm{y}){\rm{d}}s(\bm{y}).
\end{equation}
Moreover, $\mathcal{T}\mathcal{T}^*=\mathcal{T}^*\mathcal{T}=I$. 
\end{lemma}

\subsection{Factorization of $\mathcal{T}N$}
In this subsection, we multiply the near-field operator $N$ with the OtI operator $\mathcal{T}$ and give a factorization of 
$\mathcal{T}N$. First, we introduce the single-layer integral operator $S$ and single-layer potential $V$.
\begin{equation*}
(S\bm{\psi})(\bm{x})=\int_{\partial D} \Pi(\bm{x},\bm{y})  \bm{\psi}(\bm{y}){\rm{d}}s(\bm{y}), \qquad \bm{x}\in  \partial D,
\end{equation*}
\begin{equation*}
(V\bm{\psi})(\bm{x})=\int_{\partial D} \Pi(\bm{x},\bm{y})  \bm{\psi}(\bm{y}){\rm{d}}s(\bm{y}), \qquad \bm{x}\in   \mathbb{R}^3.
\end{equation*}

Let the operators $H, H^*$ and $G$ be defined as in Section \ref{sec-solution}. The adjoint operator $H^*$ can be factorized in terms of $\mathcal{T}, G$ and $S$ as follows.
\begin{theorem}\label{h}
 $H^*=\mathcal{T}GS$.
\end{theorem}
\begin{proof}
Using the expansion of $\Pi$ in (\ref{pi-expansion}), we obtain an expression of $ \bm{v}_{\bm{g}}^{in} $ as follows:
\begin{equation}
\bm{v}_{\bm{g}}^{in}(\bm{x})=i\sum_{n=0}^{\infty}\sum_{m=-n}^{n}  M_n^m(\widehat{\bm{x}}) B_n(|\bm{x}|)\ E\  [ A_n(R)]^{\top} \int_{S_R}\overline{ M_n^m(\widehat{\bm{y}})} ^\top{\bm{g}}(\bm{y}){\rm{d}}s(\bm{y}),\qquad |\bm{x}|<R,
\end{equation}
which together with the definition of $H$ implies 
\begin{equation} 
(H{\bm{g}})(\bm{x})=i\sum_{n=0}^{\infty}\sum_{m=-n}^{n}  M_n^m(\widehat{\bm{x}}) B_n(|\bm{x}|)\ E\  [ A_n(R)]^{\top} \int_{S_R}\overline{ M_n^m(\widehat{\bm{y}})} ^\top{\bm{g}}(\bm{y}){\rm{d}}s(\bm{y}), \qquad \bm{x}\in  \partial D.
\end{equation}
Since $j_n$ is real-valued, the adjoint operator $H^*$ takes the form 
\begin{equation} 
(H^*\bm{\psi})(\bm{x})=-i\sum_{n=0}^{\infty}\sum_{m=-n}^{n}   M_n^m(\widehat{\bm{x}})   \overline{A_n(R) }   \int_{\partial D}  E^\top\left[ B_n(|\bm{y}|) \right]^\top  \overline{M_n^m(\widehat{\bm{y}})}^\top\bm{\psi}(\bm{y}){\rm{d}}s(\bm{y}), \qquad \bm{x}\in  S_R,
\end{equation}
for $\bm{\psi}\in [H^{-\frac{1}{2}}(\partial D)]^3$. Obviously, $V\bm{\psi}$ is a radiating solution to the Navier equation (\ref{Navier-equation}). It follows from the expansion of $\Pi$ in (\ref{pi-expansion}) that for $ |\bm{x}|\geq R $
\begin{equation*}
(V\bm{\psi})(\bm{x})=i\sum_{n=0}^{\infty}\sum_{m=-n}^{n}   M_n^m(\widehat{\bm{x}})   A_n(|\bm{x}|)    \int_{\partial D}  E^\top\left[ B_n(|\bm{y}|) \right]^\top  \overline{M_n^m(\widehat{\bm{y}})}]^\top\bm{\psi}(\bm{y}){\rm{d}}s(\bm{y}) .
\end{equation*}
Since $(V\bm{\psi})(\bm{x})|_{\partial D}=(S\bm{\psi})(\bm{x})$, it's obvious that
\begin{equation*}
(GS\bm{\psi})(\bm{x})=i\sum_{n=0}^{\infty}\sum_{m=-n}^{n}   M_n^m(\widehat{\bm{x}})   A_n(R)    \int_{\partial D}  E^\top\left[ B_n(|\bm{y}|) \right]^\top  \overline{M_n^m(\widehat{\bm{y}})}]^\top\bm{\psi}(\bm{y}){\rm{d}}s(\bm{y}) .
\end{equation*}
Recalling the definition of $\mathcal{T}$, we can verify that $H^*=\mathcal{T}GS$. The proof is complete.
\end{proof} 

By Theorem \ref{h} and the definition of the above operators,  the near-field operator can be factorized as follows.

\begin{theorem}
(i) The near-field operator $N$ has the factorization 
\begin{equation}\label{fac-N}
\mathcal{T}N=-\mathcal{G}S^*\mathcal{G}^*,\qquad \mathcal{G}:=\mathcal{T}G.
\end{equation}
(ii) The operator $\mathcal{G}: [H^{\frac{1}{2}}(\partial D)]^3\rightarrow [L^2(S_R)]^3 $ is compact with a dense range in $[L^2(S_R)]^3 $.
\end{theorem} 
\begin{proof}
(i) By (\ref{factorization1}), we have $N=-GH.$ By Theorem \ref{h}, $H=S^*G^*\mathcal{T}^*=S^*\mathcal{G}^*.$ Hence $\mathcal{T}N=-\mathcal{T}GH=-\mathcal{T}GS^*G^*\mathcal{T}^*=-\mathcal{G}S^*\mathcal{G}^*.$ The proof is complete.\\
(ii) As $\mathcal{T}\mathcal{T}^*=\mathcal{T}^*\mathcal{T}=\mathbf{I}$, the operator $\mathcal{T}$ is bounded and has a dense range in $[L^2(S_R)]^3 $. Together with Theorem \ref{G-com}, one can find that  $ \mathcal{G}=\mathcal{T}G $ is compact with a dense range in $[L^2(S_R)]^3 $.
\end{proof}

Throughout this paper, we define the function $ \bm{\phi}^{\bm{a}}_{\bm{z}}(\cdot)=\overline{\Pi(\cdot,\bm{z})} \bm{a}|_{S_R}\in [L^2(S_R)]^3,\ \bm{a}\in \mathbb{S}^2.$
\begin{lemma}\label{range}
For $\bm{z}\in B_R $ and $\bm{a}\in \mathbb{S}^2$, $\bm{z}\in D$ if and only if $\bm{\phi}^{\bm{a}}_{\bm{z}}$ belongs to the range of $\mathcal{G}$.
\end{lemma}
\begin{proof}
For $\bm{z}\in D$, the function $\bm{u}(\bm{x}):=\Pi(\bm{x},\bm{z}) \bm{a}, \bm{x}\in  D^c $, is the unique radiating solution to the scattering problem (\ref{Navier-equation})-(\ref{radiation}) with the boundary condition $\bm{f}(\bm{x}):=\Pi(\bm{x},\bm{z}) \bm{a}|_{ \partial D}.$ It follows form the definition of $\mathcal{G}$ and Lemma \ref{T-character} that $(\mathcal{G}\bm{f})(\bm{x})=\overline{\Pi(\bm{x},\bm{z})} \bm{a}, \bm{x}\in S_R.$ That is $\bm{\phi}^{\bm{a}}_{\bm{z}}\in \text{Range}(\mathcal{G}).$

On the other hand, let $\bm{z}\in  D^c $ and assume that $(\mathcal{G}\bm{f})=\bm{\phi}^{\bm{a}}_{\bm{z}}$ for some $\bm{f} \in  [H^{\frac{1}{2}}(\partial D)]^3.$ Let $ \bm{v} $ be the unique radiating solution to the scattering problem (\ref{Navier-equation})-(\ref{radiation}) with the boundary condition $\bm{v}=\bm{f}$ on $ \partial D.$ Since the operator $\mathcal{T}$ is injective,  $(\mathcal{G}\bm{f})=\bm{\phi}^{\bm{a}}_{\bm{z}}$ means that $(G\bm{f})(\bm{x})=\Pi(\bm{x},\bm{z}) \bm{a}|_{S_R}$. Thus, we get $\bm{v}(\bm{x})=\Pi(\bm{x},\bm{z}) \bm{a}$ for $|\bm{x}|\geq R$ due to the uniqueness of the exterior Dirichlet problem. Therefore $\bm{v}(\bm{x})=\Pi(\bm{x},\bm{z}) \bm{a}$ in $ D^c\setminus \{\bm{z}\}$, which contradicts with the analyticity of $\bm{v}$ in $  D^c  $. At last, if $\bm{z} \in \partial D$, $\Pi(\cdot,\bm{z}) \bm{a}|_{\partial D} \notin [H^{\frac{1}{2}}(\partial D)]^3 $ contradicts with the fact that $\bm{v}|_{\partial D} \in [H^{\frac{1}{2}}(\partial D)]^3 $. Hence, $\bm{z}\in D.$
\end{proof}

Next we give some properties of the single-layer operator $S$ (see \cite[Lemmas 6.1, 6.2]{A-K2002}), which extend the results \cite[Lemma 5.37]{Colton2013}  from acoustic scattering to the elastic case in $3D$.
\begin{lemma}\label{single-property}
Assume that $\omega^2$ is not the Dirichlet eigenvalue of the operator $-\Delta^*$ in $D$. Then\\
(i) The single-layer operator $S: [H^{-{\frac{1}{2}}}(\partial D)]^3\rightarrow [H^{\frac{1}{2}}(\partial D)]^3 $ is an isomorphism.\\
(ii) $Im \langle\bm{\psi}, S \bm{\psi}\rangle <0$ for $\bm{\psi} \in [H^{-{\frac{1}{2}}}(\partial D)]^3$ with $\bm{\psi}\neq 0$. Here, $\langle \cdot, \cdot\rangle$ denotes the duality between \\
 \indent $[H^{-{\frac{1}{2}}}(\partial D)]^3$ and $ [H^{\frac{1}{2}}(\partial D)]^3 $.\\
(iii) Let $S_i: H^{-{\frac{1}{2}}}(\partial D)^3\rightarrow H^{\frac{1}{2}}(\partial D)^3 $ be the single-layer operator for the frequency $\omega=i$. Then\\
\indent $S_i$ is compact, self-adjoint, and coercive.\\
(iv) The difference $S-S_i:  [H^{-{\frac{1}{2}}}(\partial D)]^3\rightarrow [H^{\frac{1}{2}}(\partial D)]^3 $ is compact. 
\end{lemma}

Relying on properties of the operators $\mathcal{G}$ and $S$ (see Lemmas \ref{range} and \ref{single-property}), the following range identity \cite[Theorem 2.15]{A-N2008} can be applied to the operator $\mathcal{T}N$. For an operator $F$, the operator$ F_{\sharp} $ is defined as $F_{\sharp}:=\vert Re[\exp(it)F]\vert + \vert Im F\vert$.

\begin{lemma}\label{range-identity}
 Let $X\subset U \subset X^*$ be a Gelfand triple with a Hilbert space $U$ and a reflexive Banach space $X$ such that the imbedding  is dense. Furthermore, let $Y$ be a second Hilbert space and let $F : Y \rightarrow Y, G: X \rightarrow Y$  and $T : X^* \rightarrow X$ be linear bounded operators such that $F=GTG^*$. Assume that  \\
(Al) $G$ is compact with dense range.  \\
(A2) There exists $t\in [0, 2\pi]$ such that $Re[\exp(it)T] $ has the form $Re[\exp(it)T] = T_0+T_1$ with \\
\indent some compact operator $T_1$ and some self-adjoint and coercive operator $T_0: X^*\rightarrow X$, i.e., \\\indent there exists $c > 0$

\begin{equation*}
 \langle \varphi, T_0\varphi \rangle \geq c||\varphi||^2 \qquad  \text{for all} \quad \varphi \in X^*.
\end{equation*}
(A3) $Im T$ is compact and non-negative on $\text{Range}(G^*) \subset X^*$, i.e., $\langle \varphi, (Im T)\varphi \rangle \geq 0$ for all\\
\indent $\varphi \in \text{Range}(G*)$.   \\
(A4) $Re[\exp(it)T] $ is one-to-one or $Im T$ is strictly positive on the closure $\overline{\text{Range}(G^*)}$ of \\
\indent $\text{Range}(G^*)$, i.e.,$\langle \varphi, (Im T)\varphi \rangle)> 0$ for all $\varphi \in \overline{\text{Range}(G^*)}$  with $\varphi\neq 0$.

Then the operator $F_{\sharp}$ is positive, and the ranges of $G : X \rightarrow Y$ and $F_{\sharp}^{1/2}: Y\rightarrow Y$ coincide.  
\end{lemma}
 
Applying Lemma \ref{range-identity}, we have the following result.
\begin{theorem}\label{result}
Assume that $\omega^2$ is not the Dirichlet eigenvalue of the operator $-\Delta^*$ in $D$. Let $\bm{z}\in B_R $ and $\bm{\phi}^{\bm{a}}_{\bm{z}}$ be given in Lemma \ref{range}. Denote by $\lambda_j \in \mathbb{C} $ the eigenvalues of the operator\\
 $(\mathcal{T}N)_{\sharp}$ with the corresponding normalized eigenfunctions  $\bm{\psi}_j\in [L^2(S_R)]^3$.  Then\\
(i) $\bm{z}\in D$ if and only if $\bm{\phi}^{\bm{a}}_{\bm{z}}$ belongs to the range $\text{Range}\left( (\mathcal{T}N)_{\sharp} ^{1/2}\right)$ of $ (\mathcal{T}N)_{\sharp} ^{1/2}$.\\
(ii) $\bm{z}\in D$ if and only if 
\begin{equation}\label{function-W}
W^{\bm{a}}(\bm{z}):=\left[\sum_{j=1}^{\infty}\dfrac{\vert\langle \bm{\phi}^{\bm{a}}_{\bm{z}}, \bm{\psi}_j\rangle_{[L^2(S_R)]^3}\vert^2}{\vert  \lambda_j \vert}\right]^{-1}>0.
\end{equation}
\end{theorem}
\begin{proof}
 Set $t=0, F=\mathcal{T}N, G=\mathcal{G}, T=S^*, T_0=S_i, T_1=Re(S-S_i), X= H^{\frac{1}{2}}(\partial D)^3$ and $Y=[L^2(S_R)]^3 $. Then application of Lemmas \ref{range}, \ref{single-property} and \ref{range-identity} gives (i).
 (ii) follows from Picard's range criterion. 
\end{proof} 
 
By Lemma \ref{range}, one can choose any vector $\bm{a}\in \mathbb{S}^2$ in Theorem \ref{result}. For multiple vectors $\{\bm{a}_i\in \mathbb{S}^2: j=1,2,3\}$, we define the following indicator function: 
\begin{equation}\label{function-W}
W(\bm{z})=\left[\sum_{j=1}^{3}\frac{1}{W^{\bm{a}_j}(\bm{z})} \right]^{-1},\qquad  \bm{z}\in B_R.
\end{equation}

\begin{lemma}
 $\bm{z}\in D$ if and only if $W(\bm{z})>0.$
\end{lemma}
\begin{proof}
If $\bm{z}\in D$, Theorem \ref{result} implies that $0<W^{\bm{a}_j}(\bm{z})<\infty$ for all $j=1,2,3,$ and thus $W(\bm{z})>0.$ On the other hand, if $\bm{z} \notin D$, again applying Theorem \ref{result} gives $ W^{\bm{a}_j}(\bm{z})=0$ for all $j=1,2,3,$ implying that $W(\bm{z})=0.$
\end{proof}

\section{Numerical experiments in $ \mathbb{R}^2$} \label{Numeric}
In this section, we present some numerical examples in $\mathbb{R}^2$ to illustrate the applicability and effectiveness    of our inversion schemes.
\subsection{OtI operator in 2D}
We first give the outgoing-to-incoming operator $\mathcal{T}$ in $\mathbb{R}^2$. By employing the polar coordinates we write $\bm{x}=(r\cos \theta_{\bm{x}}, r\sin \theta_{\bm{x}} )$. According to the Hodge decomposition, we can introduce scalar functions $\varphi(r, \theta_{\bm{x}})$ and $ \psi(r, \theta_{\bm{x}}) $ such that the scattered field $\bm{u}$ has the form
\begin{equation}\label{2-scatteredu}
\bm{u}(r, \theta_{\bm{x}})=\text{grad}\ \varphi(r, \theta_{\bm{x}})+ \text{curl}\ \psi(r, \theta_{\bm{x}}),
\end{equation}
where $\text{curl}\ f=(-\partial_2 f, \partial_1 f)^\top$ and $\varphi(r, \theta_{\bm{x}})$, $ \psi(r, \theta_{\bm{x}}) $ are radiating solutions to the Helmholtz equation.

Recall the relationship between the Cartesian and polar coordinates for gradient:
\begin{equation*}
\begin{pmatrix}
\partial_1 \\
\partial_2
\end{pmatrix}=\begin{pmatrix}
\cos \theta_{\bm{x}}&-\frac{1}{r}\sin \theta_{\bm{x}}\\
 \sin \theta_{\bm{x}}&\frac{1}{r}\cos \theta_{\bm{x}}
\end{pmatrix}\begin{pmatrix}
\partial_r \\
\partial_{\theta}
\end{pmatrix}.
\end{equation*}
Then radiating solutions $\varphi(r, \theta_{\bm{x}})$ and $ \psi(r, \theta_{\bm{x}}) $ can be expanded into
\begin{equation}\label{2-expansion}
\varphi(r, \theta_{\bm{x}})=\sum_{n\in \mathbb{Z}}\alpha_nH_n^{(1)}(k_pr)e^{in\theta_{\bm{x}}},\ \ \psi(r, \theta_{\bm{x}})=\sum_{n\in \mathbb{Z}}\beta_n H_n^{(1)}(k_sr)e^{in\theta_{\bm{x}}},
\end{equation}
where $H_n^{(1)}(k_pr),H_n^{(1)}(k_sr)$ are Hankel functions of first kind of order $n$ and $\alpha_n, \beta_n\in \mathbb{C}$ are coefficients.

Combining (\ref{2-scatteredu}) and (\ref{2-expansion}), we get

\begin{equation}\label{2-scatteredexpansion}
\bm{u}(r, \theta_{\bm{x}})=\sum_{n\in \mathbb{Z}}M(\theta_{\bm{x}})^\top A_n(r)(\alpha_n, \beta_n)^\top e^{in\theta_{\bm{x}}},
\end{equation}
where matrices $M(\theta_{\bm{x}})$ and $A_n(r)$ are defined as 
\begin{equation*}
M(\theta_{\bm{x}})=\begin{pmatrix}
\cos \theta_{\bm{x}}&\sin \theta_{\bm{x}}\\
-\sin \theta_{\bm{x}}&\cos \theta_{\bm{x}}
\end{pmatrix},\ \ A_n(r)=\begin{pmatrix}
 k_pH_n^{(1)'}(k_pr)&-inH_n^{(1)}(k_sr)/r \\
 inH_n^{(1)}(k_pr)/r& k_sH_n^{(1)'}(k_sr)
\end{pmatrix}.
\end{equation*}

The fundamental solution to the Navier equation in $\mathbb{R}^2$ is
\begin{equation}\label{2-fundamental}
 \Pi(\bm{x},\bm{y})=\frac{1}{\mu}\Phi_{k_s}(\bm{x},\bm{y})\mathbf{I}+\frac{1}{\omega^2}\nabla_{\bm{x}}\nabla_{\bm{x}}^\top[\Phi_{k_s}(\bm{x},\bm{y})-\Phi_{k_p}(\bm{x},\bm{y})],\qquad \bm{x}\neq \bm{y},
 \end{equation}
where $\mathbf{I}$ stands for the unit matrix and $\Phi_{k}(\bm{x},\bm{y})=\dfrac{i}{4}H_0^{(1)}(k|\bm{x}-\bm{y}|)$ is the free-space fundamental solution to the Helmholtz equation $(\Delta + k^2) u=0$ in $ \mathbb{R}^2$.

Recall the definition of  OtI operator $\mathcal{T}$ in $\mathbb{R}^3$, the two-dimensional outgoing-to-incoming operator $\mathcal{T}$ can be represented as  

\begin{equation}\label{OtI2}
(\mathcal{T}\bm{\varphi})(\bm{x}):=\int_{S_R}  K(\bm{x},\bm{y})\bm{\varphi}(\bm{y}){\rm{d}}s(\bm{y}), \ \ \ \bm{\varphi}(\bm{x})\in [L^2(S_R)]^2,
\end{equation}
where
\begin{equation*}
B_n(R)=\begin{pmatrix}
 k_p\overline{H_n^{(1)'}(k_pR)}&-in\overline{H_n^{(1)}(k_sR)}/R \\
 in\overline{H_n^{(1)}(k_pR)}/r& k_s\overline{H_n^{(1)'}(k_sR)}
\end{pmatrix}
\end{equation*}
and the integral kernel $K(\bm{x},\bm{y})$ is defined as
\begin{equation*}
K(\bm{x},\bm{y}):=-\frac{1}{2\pi R}\sum_{n\in \mathbb{Z}} \left[ M(\theta_{\bm{x}})\right]^\top B_n(R)A_n(R)^{-1}  M(\theta_{\bm{y}})e^{in(\theta_{\bm{x}}-\theta_{\bm{y}})}.
\end{equation*}
In our numerical implementation, the operator $\mathcal{T}$ is approximated by the truncated operator
 \begin{equation*}
 (\mathcal{T}_{M_1}\bm{\varphi})(\bm{x}):=\int_{S_R}  K_{M_1}(\bm{x},\bm{y})\bm{\varphi}(\bm{y}){\rm{d}}s(\bm{y}), \ \ \ \bm{\varphi}(\bm{x})\in [L^2(S_R)]^2,
 \end{equation*}
with the kernel given by
\begin{equation*}
K_{M_1}(\bm{x},\bm{y}):=-\frac{1}{2\pi R}\sum_{n=-M_1}^{n=M_1} \left[ M(\theta_{\bm{x}})\right]^\top B_n(R)A_n(R)^{-1}  M(\theta_{\bm{y}})e^{in(\theta_{\bm{x}}-\theta_{\bm{y}})},
\end{equation*}
for some $M_1\in \mathbb{Z}^{+}$.
\subsection{Numerical examples}
In our simulations, we used the boundary integral equation method to compute the near-field data on $S_R.$ Let $\bm{a}_1=(1,0)^\top$ and $\bm{a}_2=(0,1)^\top$ denote two incident polarization vectors. Define the step size $\triangle \theta_{\bm{x}}=\triangle \theta_{\bm{y}}=2\pi/M_2$ for some $M_2\in \mathbb{Z}^{+}$, that is 
\begin{equation*}
\theta_{\bm{x}_j}=(j-1)\triangle \theta_{\bm{x}},\qquad \theta_{\bm{y}_j}=(j-1)\triangle \theta_{\bm{y}},\qquad j=1,2,\ldots,M_2.
\end{equation*}
We choose $2M_2$ incident point sources as 
\begin{equation*}
\mathcal{K}=\{\Pi(\bm{x},\bm{y}_j)\bm{a}_1, \Pi(\bm{x},\bm{y}_j)\bm{a}_2: j=1,2,\ldots,M_2\},
\end{equation*}
where $\bm{y}_j=R(\cos \theta_{\bm{y}_j}, \sin \theta_{\bm{y}_j})$. Denote by $\bm{u}(\bm{x},\bm{y}_j; \bm{a}_1)$ and $\bm{u}(\bm{x},\bm{y}_j; \bm{a}_2)$ the scattered field corresponding to the incident field $\Pi(\bm{x},\bm{y}_j)\bm{a}_1$ and $\Pi(\bm{x},\bm{y}_j)\bm{a}_2$, respectively. Then we have the near-field matrix
\begin{equation*}
\bm{N}_{M_2\times 2M_2}=[\bm{b}_{1,1},\bm{b}_{1,2},\bm{b}_{2,1},\bm{b}_{2,2},\ldots,\bm{b}_{M_2,1},\bm{b}_{M_2,1}],
\end{equation*}
where
\begin{align*}
\bm{b}_{j,1}:=\left(\bm{u}(\bm{x}_1,\bm{y}_j; \bm{a}_1),\bm{u}(\bm{x}_2,\bm{y}_j; \bm{a}_1),\ldots, \bm{u}(\bm{x}_{M_2},\bm{y}_j; \bm{a}_1)\right)^\top\in \mathbb{C}^{ M_2\times 1},\\
  \bm{b}_{j,2}:=\left(\bm{u}(\bm{x}_1,\bm{y}_j; \bm{a}_2),\bm{u}(\bm{x}_2,\bm{y}_j; \bm{a}_2),\ldots, \bm{u}(\bm{x}_{M_2},\bm{y}_j; \bm{a}_2)\right)^\top\in \mathbb{C}^{ M_2\times 1}.
\end{align*}
The outgoing to incoming operator $\mathcal{T}$ is approximated by the following matrix
\begin{equation*}
\bm{\mathcal{T}}_{M_2\times M_2}=\left( K_{M_1}(\bm{x}_i,\bm{y}_j)\right)_{ M_2\times M_2},
\end{equation*}
where $K_{M_1}(\bm{x}_i,\bm{y}_j)$ is defined as 
\begin{equation*}
K_{M_1}(i,j):=-\frac{1}{2\pi R}\sum_{n=-M_1}^{n=M_1} \left[ M(\theta_{\bm{x}_i})\right]^\top B_n(R)A_n(R)^{-1}  M(\theta_{\bm{y}_j})e^{in(\theta_{\bm{x}_i}-\theta_{\bm{y}_j})},\ i,j=1,2,\ldots,M_2.
\end{equation*}
Letting $\bm{\mathcal{T}N_{M_2\times 2M_2}}:=\bm{\mathcal{T}}_{M_2\times M_2}\bm{N}_{M_2\times 2M_2}$, the characteristic function $W$ defined in (\ref{function-W}) can be approximated by the series
\begin{equation*}
W^{\bm{a}}_{M_2}(\bm{z})=\left[\sum_{j=1}^{M_1}\dfrac{\vert\langle \bm{\phi}^{\bm{a}}_{\bm{z}}, \bm{\psi}_j\rangle_{[L^2(S_R)]^3}\vert^2}{\vert  \lambda_j \vert}\right]^{-1},
\end{equation*}
where $ \bm{\phi}^{\bm{a}}_{\bm{z}}=\{\overline{\Pi(\bm{x}_1,\bm{z})} \bm{a}, \overline{\Pi(\bm{x}_2,\bm{z})} \bm{a}, \ldots, \overline{\Pi(\bm{x}_{M_2},\bm{z})} \bm{a}\}^\top$ with $\bm{a}=(\cos \alpha, \sin \alpha)^\top$ and $\{ \bm{\psi}_j, \lambda_j\}_{j=1}^{M_2}$
is an eigensystem of the matrix $\bm{\mathcal{T}N} :=|Re(\bm{\mathcal{T}N} )|+|Im(\bm{\mathcal{T}N} )|.$ 
To test the stability, we add relative noise to the data
\begin{equation*}
% \bm{N}_{M_2\times 2M_2}^{\delta}=\bm{N}_{M_2\times 2M_2}(1+%\delta \text{rand})
\bm{u}^{\delta}(\bm{x}_i,\bm{y}_j; \bm{a}_s)=\bm{u}(\bm{x}_i,\bm{y}_j; \bm{a}_s)(1+\delta\, \text{rand}),\qquad i,j=1,2,\ldots M_2,\ s=1,2.
 \end{equation*} 
where rand are uniformly distributed random numbers in $[-1,1]$. 

In the following experiments, we set $\lambda=2, \mu=1, \omega=10, R=4, M_1=40$ and $M_2=64$. We choose two polar vectors as $\bm{a}_1=(1,0)^\top$ and $\bm{a}_2=(0,1)^\top.$

\noindent {\bf{Example 1}}: Consider a kite-shaped obstacle, which has the parametric equation
\begin{equation*}
 \bm{x}(t)= (\cos t +0.65\cos(2t)-0.65,\ 1.5\sin t ),\qquad t\in[0,2\pi].
\end{equation*}
\begin{figure}
	\centering
	\subfigure[The kite-shaped obstacle]{
		\includegraphics[width=2in,height=1.5in]{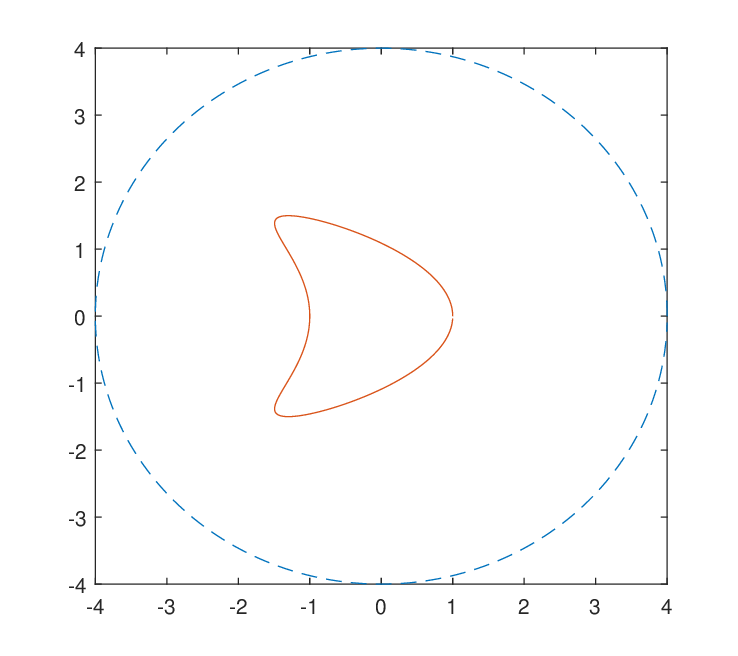}
	}	
 	\subfigure[No noise]{
 		\includegraphics[width=2in,height=1.5in]{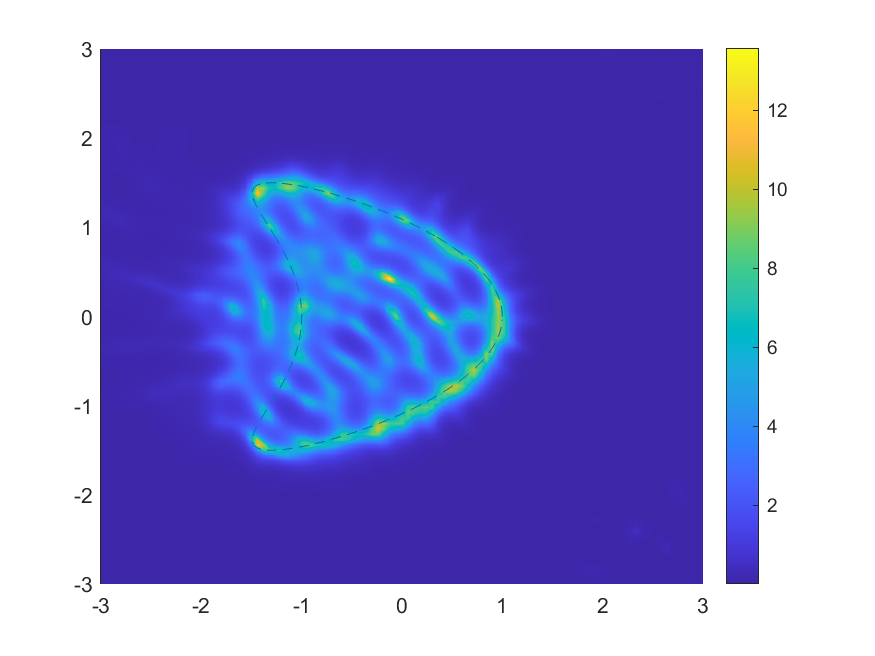 }	
}
 	\subfigure[ $5\% $  noise]{
		\includegraphics[width=2in,height=1.5in]{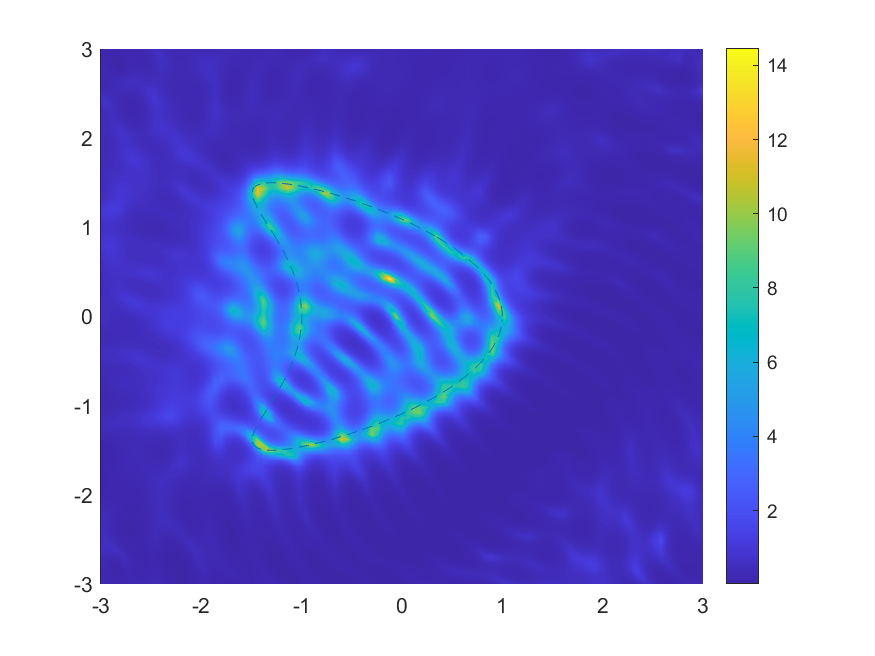}
	}	
 	\subfigure[$10\% $ noise]{
 		\includegraphics[width=2in,height=1.5in]{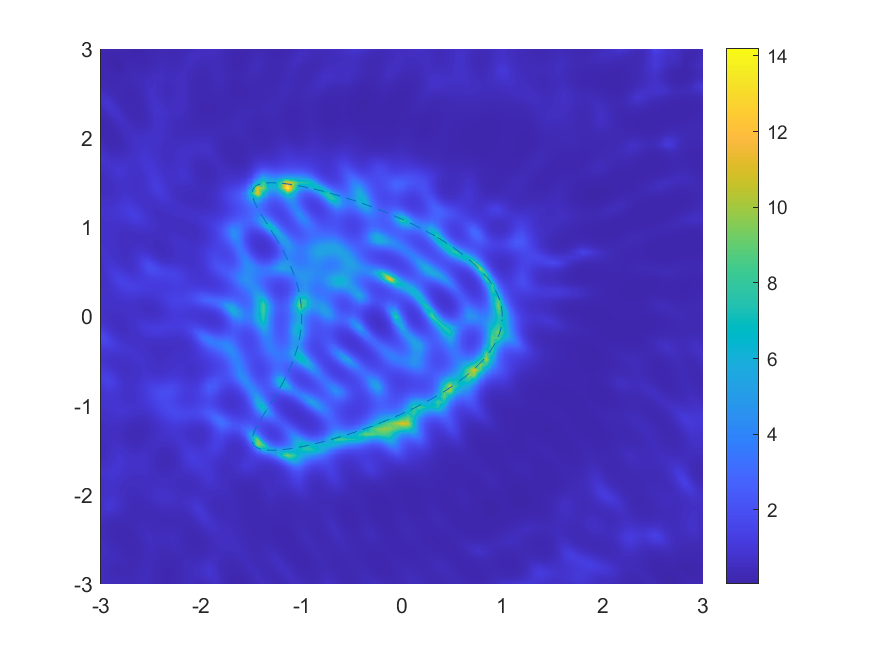 }
 			}

\caption{Reconstruction of the kite-shaped obstacle with $\bm{a}=(\cos \alpha, \sin \alpha)^\top, \alpha=\frac{2}{3}\pi$ . }
\label{fig:kite-noise} 
\end{figure}
Figure \ref{fig:kite-noise} presents reconstruction results from the near-field data with the polarization vector\\ $\bm{a}=(\cos \alpha, \sin \alpha)^\top, \alpha=\frac{2}{3}\pi$ and noise level $\delta=0\%$ in Figure \ref{fig:kite-noise}(b), $\delta=5\%$ in Figure \ref{fig:kite-noise}(c), and $\delta=10\%$ in figure \ref{fig:kite-noise}(d). In Figure \ref{fig:kite-a}, we set noise level $\delta=5\%$ to see the reconstruction results with different choice of $\bm{a}=(\cos \alpha, \sin \alpha)^\top.$ The polarization angel $\alpha$ is chosen to be zero in Figure \ref{fig:kite-a}(a), $1/2\pi$ in Figure \ref{fig:kite-a}(b) and $2/3\pi$ in Figure \ref{fig:kite-a}(c). Figure \ref{fig:kite-a}(d) is a superposition of the results for $\alpha=0, 1/2\pi$ and $2/3\pi$.

\noindent {\bf Example 2}: $D$ is given as a  star-shaped obstacle with the boundary parameterized by
\begin{equation*}
 \bm{x}(t)= (1+0.2\cos(5t))\ (\cos t ,\ \sin t ),\qquad t\in[0,2\pi].
\end{equation*}
Figure \ref{fig:star-noise} presents the reconstruction results from the true data and noised data at levels $5\% $ and $10\% $, respectively. Figure \ref{fig:star-a} shows the results from the data with noise levels at $10\% $ and different polarization vectors $\bm{a}=(\cos \alpha, \sin \alpha)^\top.$ 

\begin{figure} 
	\centering
	\subfigure[$\alpha=0$]{
		\includegraphics[width=2in,height=1.5in]{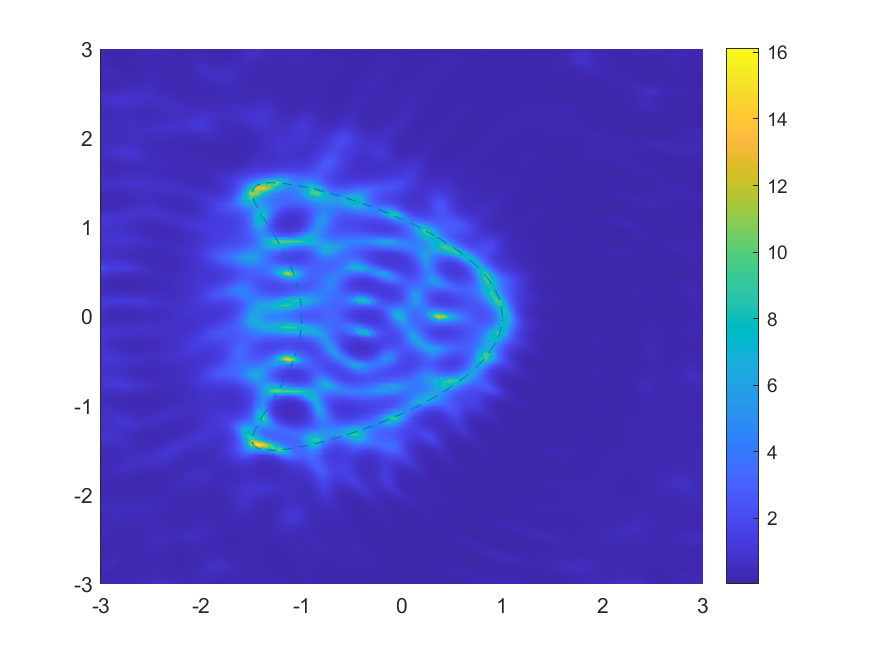}
		}
 	\subfigure[$\alpha=\dfrac{1}{2}\pi$]{
 		\includegraphics[width=2in,height=1.5in]{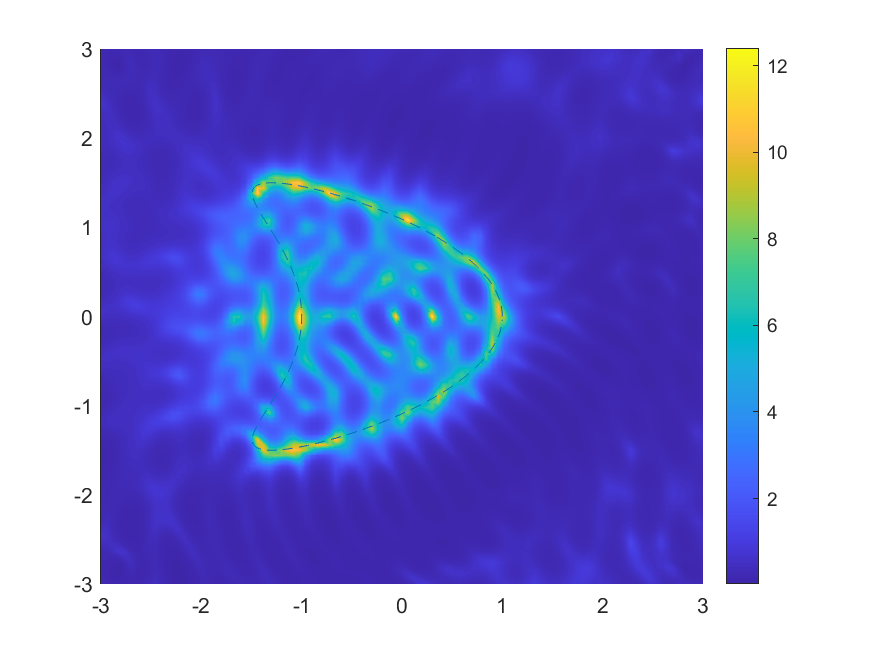 }
 	}
    \subfigure[$\alpha=\dfrac{2}{3}\pi$]{
 		\includegraphics[width=2in,height=1.5in]{Figure/kite/0.05-noise1.eps }
 	}
 	\subfigure[Multiple $\alpha$]{
 		\includegraphics[width=2in,height=1.5in]{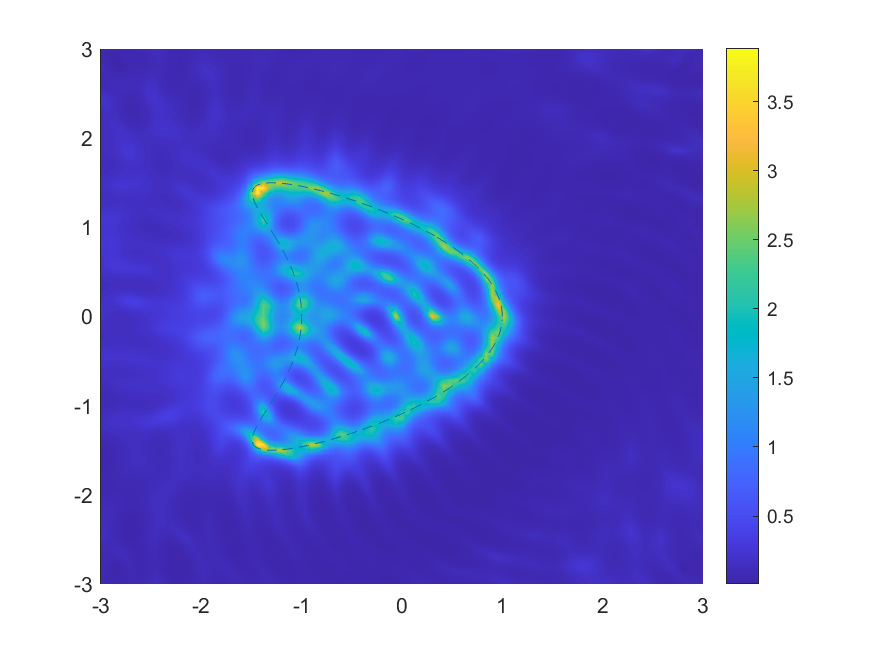 }
 	}
 	
\caption{Reconstruction of the kite-shaped obstacle with different $\alpha$.}
 \label{fig:kite-a} 
\end{figure}

\begin{figure} 
	\centering
	\subfigure[The star-shaped obstacle]{
		\includegraphics[width=2in,height=1.5in]{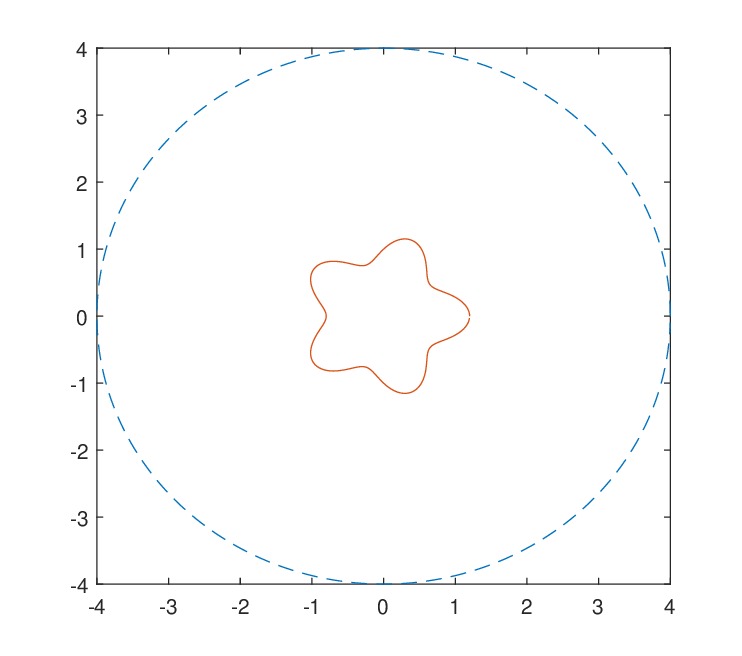}
	}	
 	\subfigure[No noise]{
 		\includegraphics[width=2in,height=1.5in]{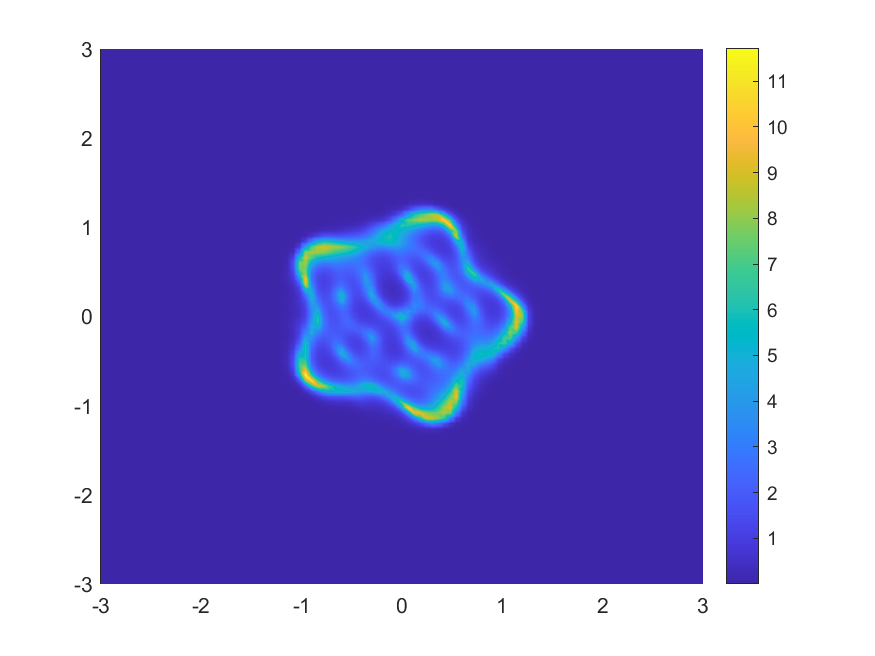 }
 		
}
 	\subfigure[ $5\% $  noise]{
		\includegraphics[width=2in,height=1.5in]{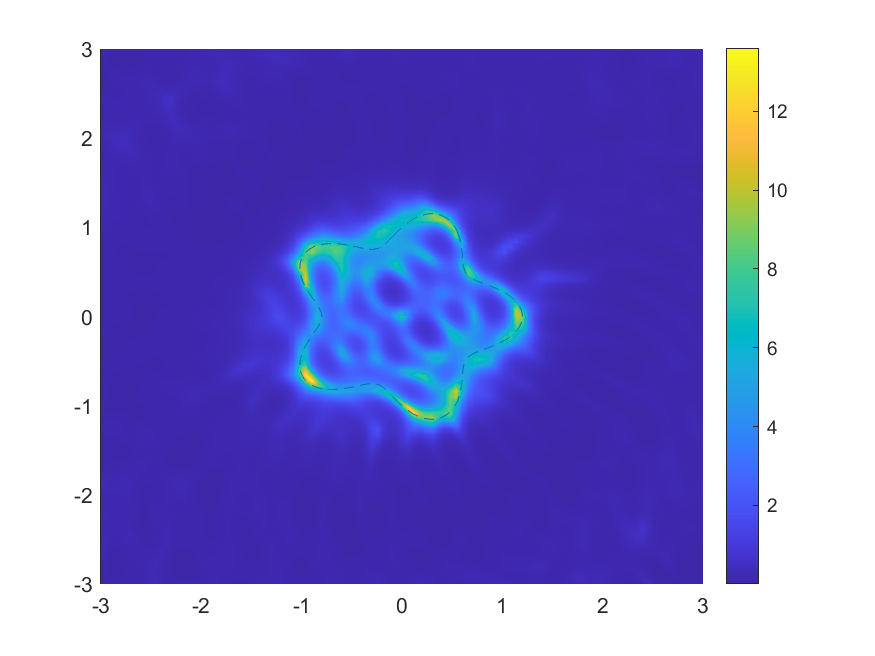}
	}	
 	\subfigure[$10\% $ noise]{
 		\includegraphics[width=2in,height=1.5in]{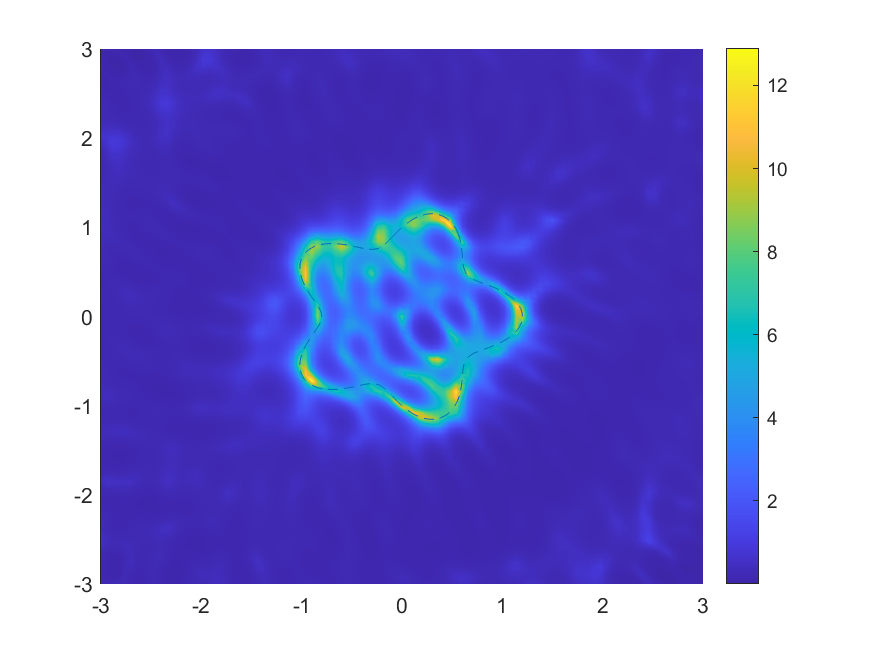 }
 			}

\caption{Reconstruction of the star-shaped obstacle with $\bm{a}=(\cos \alpha, \sin \alpha)^\top, \alpha=\frac{2}{3}\pi$.  }
\label{fig:star-noise}
\end{figure}

\begin{figure}\label{star-a} 
	\centering
	\subfigure[$\alpha=0$]{
		\includegraphics[width=2in,height=1.5in]{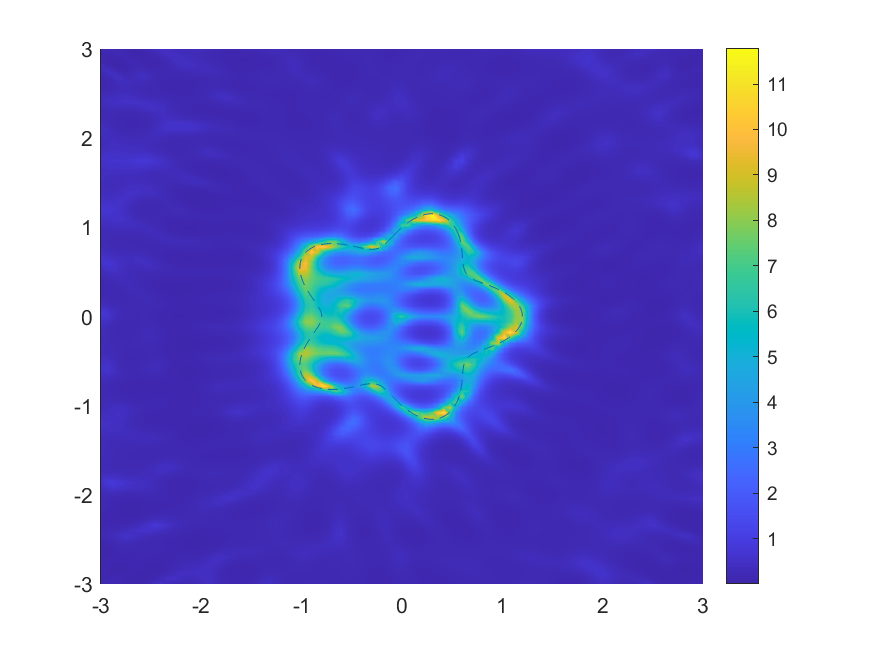}
	}	
 	\subfigure[$\alpha=\dfrac{1}{2}\pi$]{
 		\includegraphics[width=2in,height=1.5in]{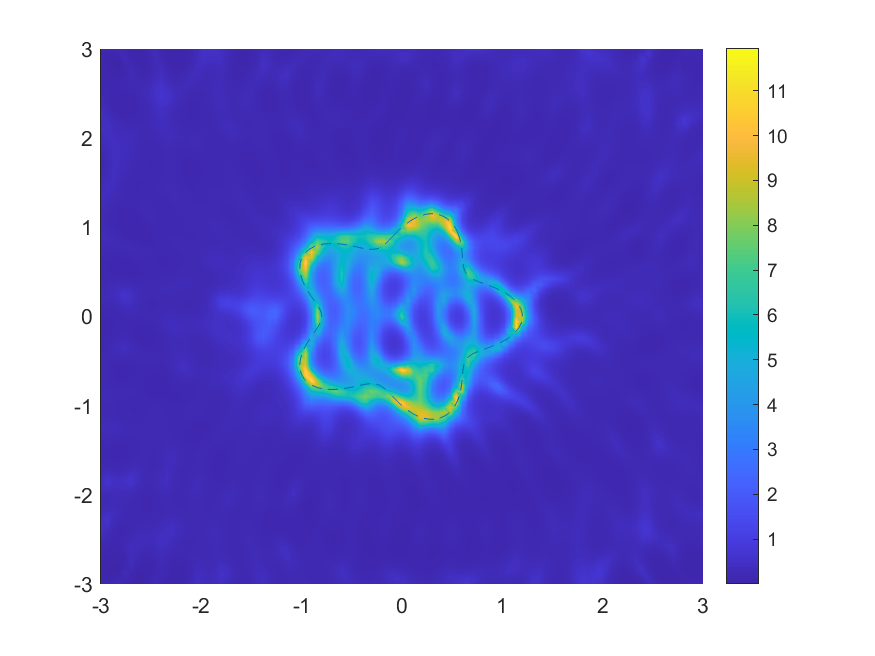 }
 	}
\subfigure[$\alpha=\dfrac{2}{3}\pi$]{
 		\includegraphics[width=2in,height=1.5in]{Figure/star/0.1-noise1.eps }
 	}
 	\subfigure[Multiple $\alpha$]{
 		\includegraphics[width=2in,height=1.5in]{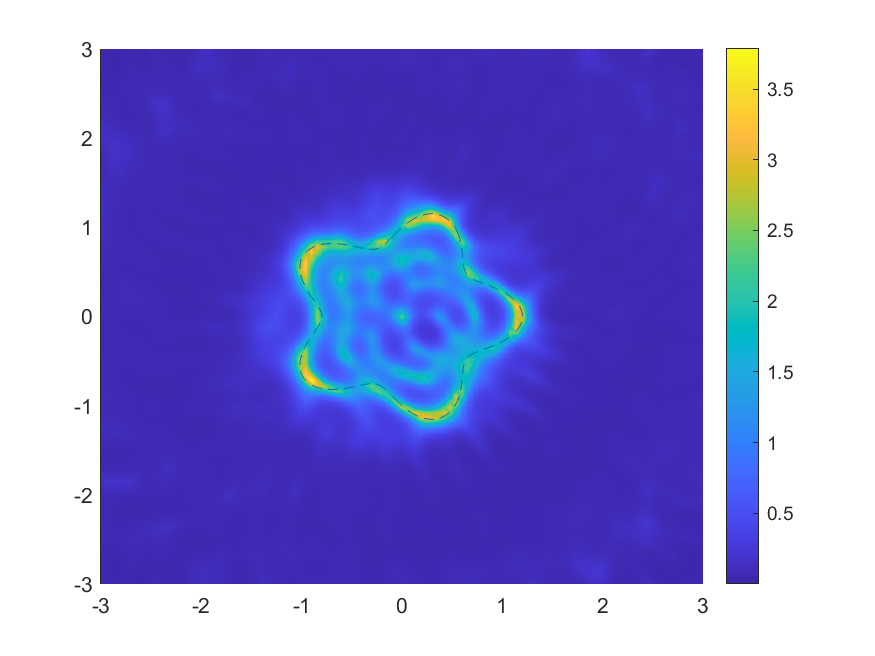 }
 	}

\caption{Reconstruction of the star-shaped obstacle with different $\alpha$.}
\label{fig:star-a}
\end{figure}     
 
\noindent {\bf Example 3}: In the third example, the obstacle $D$ is given by the union of two disjoint components $D_1$ and $D_2$. $D_1$ is a star-shaped obstacle with radius 1 centered at $(2,2)$, and $D_1$ is a kite-shaped obstacle with radius 0.5 and centered at $(-1,-1)$. Figure \ref{fig:double-a} presents the reconstruction results from the data with noise levels at $2\% $ and different polarization vectors $\bm{a}=(\cos \alpha, \sin \alpha)^\top.$ 

The numerical examples presented above demonstrate that the indicator function $W$ peaks on the boundaries of the obstacles and can reconstruct the locations and shapes. Unlike the case in acoustics, the behavior of the indicator function is influenced by the polarization vector $\bm{a}$. For a single obstacle, a single vector $\bm{a}$ suffices to reconstruct the shape and location. However, in the case of two disconnected components, the reconstructions from one polarization vector don't show complete boundary information. The third example highlights that utilizing near-field data with multiple vectors $\bm{a}$ enhances the reconstruction quality on the identification of obstacle boundaries.

\begin{figure}\label{double-a} 
	\centering
	\subfigure[$\alpha=0$]{
		\includegraphics[width=2in,height=1.5in]{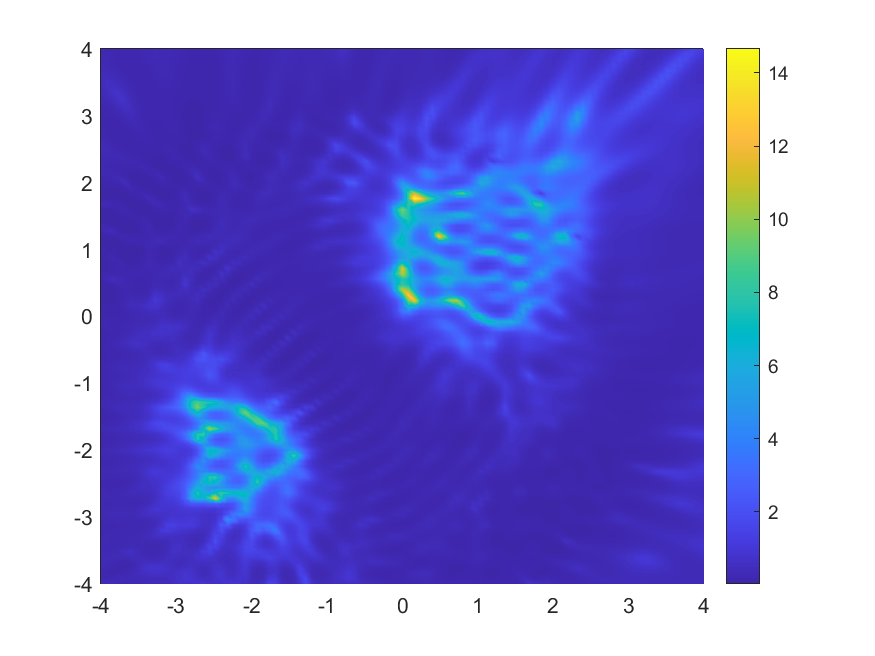}
	}	
 	\subfigure[$\alpha=\dfrac{1}{2}\pi$]{
 		\includegraphics[width=2in,height=1.5in]{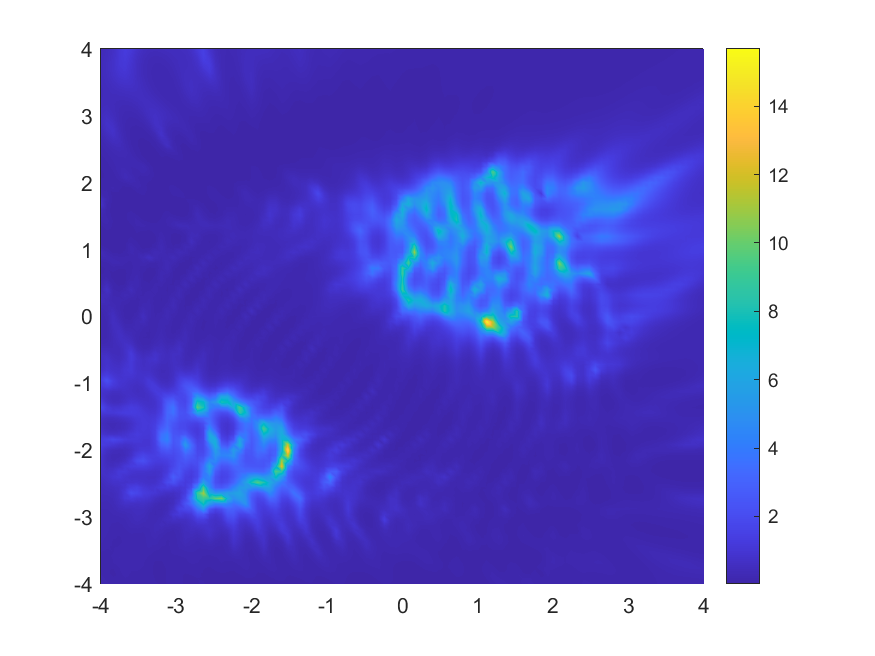 }
 	}
\subfigure[$\alpha=\dfrac{2}{3}\pi$]{
 		\includegraphics[width=2in,height=1.5in]{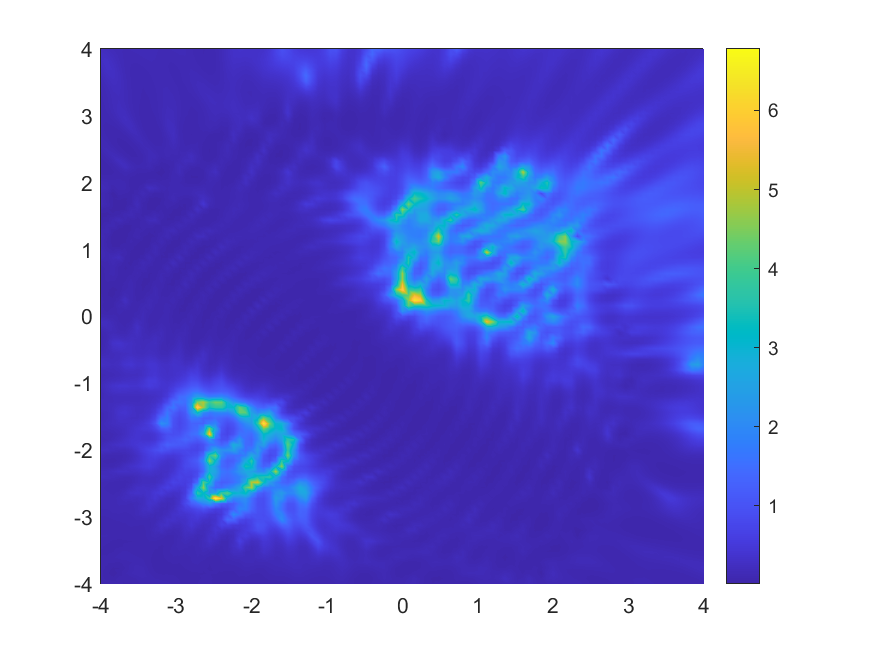 }
 	}
 	\subfigure[Multiple $\alpha$]{
 		\includegraphics[width=2in,height=1.5in]{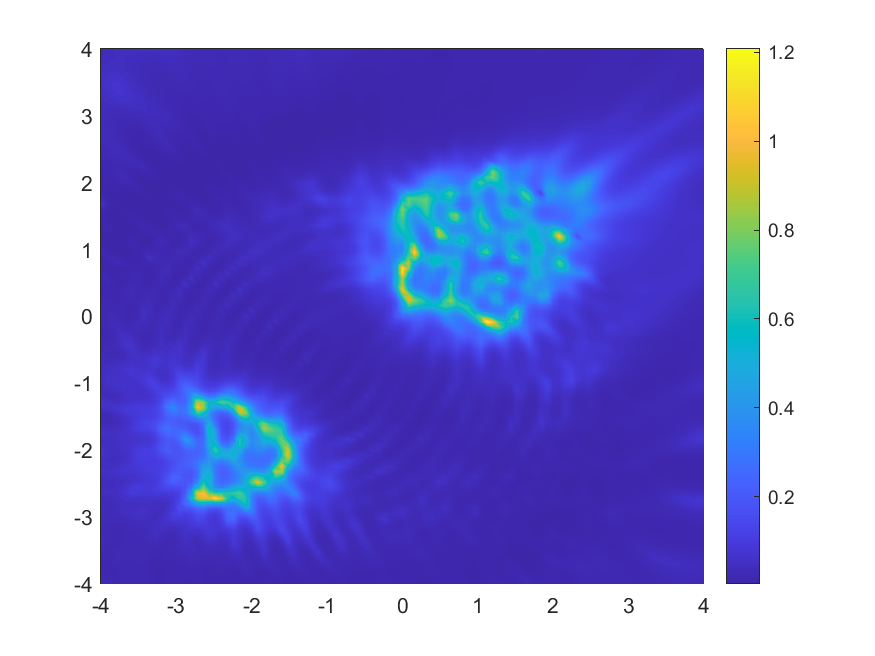 }
 	}

\caption{Reconstruction of two sound-soft obstacles with different $\alpha$.}
\label{fig:double-a}
\end{figure}

\end{document}